\documentclass[12pt]{article}
\usepackage{amsmath,amsthm}			
\usepackage{amssymb}
\usepackage{hyperref}				
\usepackage{tikz}
\usepackage{enumerate,dsfont}
\usepackage{fullpage}
\usepackage[all]{xy}
\usepackage[enableskew,vcentermath]{youngtab}
\usetikzlibrary{shapes.geometric,arrows}
\usetikzlibrary{calc,patterns}

\newtheorem{theorem}{Theorem}[section]
\newtheorem{lemma}[theorem]{Lemma}
\newtheorem{corollary}[theorem]{Corollary}
\newtheorem{proposition}[theorem]{Proposition}

\theoremstyle{definition}
\newtheorem{example}[theorem]{Example}  
\newtheorem{definition}[theorem]{Definition}   

\theoremstyle{remark}

\tikzstyle{vertex}=[circle, draw, fill=black, inner sep=0pt, minimum size=1pt]

\renewcommand\S{\mathcal S}
\newcommand\D{\mathcal D}

\newcommand\A{\mathbf A}
\newcommand\B{\mathbf B}
\newcommand\cA{\mathcal A}
\newcommand\I{\mathcal I}
\newcommand\F{\mathcal F}

\renewcommand\P{\mathsf P}
\DeclareMathOperator\fp{fp}
\DeclareMathOperator\exc{exc}
\DeclareMathOperator\wexc{wexc}
\DeclareMathOperator\Exc{Exc}

\newcommand\ant{\alpha}

\DeclareMathOperator\inv{inv}
\DeclareMathOperator\rk{rk}
\DeclareMathOperator\evac{Evac}
\DeclareMathOperator\Rvac{Rvac}
\newcommand\Path{\delta}

\newcommand\rD{\rho_\D}
\newcommand\rS{\rho_\S}
\newcommand\rA{\rho_\cA}
\newcommand\rI{\rho_\I}
\newcommand\rF{\rho_\F}

\DeclareMathOperator\LK{LK}
\newcommand\LKS{\LK_\S}
\newcommand\LKA{\LK_\cA}

\DeclareMathOperator{\sgn}{sgn}
\DeclareMathOperator{\Match}{Match}
\DeclareMathOperator{\Tab}{Tab}

\newcommand\one{\mathds{1}}
\newcommand\ind{\chi}
\def\N{-- ++(0,1)}
\def\E{-- ++(1,0)}
\newcommand\uu{\texttt{u}}
\newcommand\dd{\texttt{d}}

\newcommand{\ultriang}{
\begin{tikzpicture}[scale=.15]
\draw (0,0)--(0,1)--(1,1)--(0,0);
\end{tikzpicture}}
\newcommand{\lrtriang}{
\begin{tikzpicture}[scale=.15]
\draw (0,0)--(1,0)--(1,1)--(0,0);
\end{tikzpicture}}
\newcommand{\Dul}{\D^{\ultriang}}
\newcommand{\Dlr}{\D^{\lrtriang}}
\newcommand{\NC}{\mathcal N}
\newcommand\CSNC{\NC^S}
\newcommand{\card}[1]{{\lvert #1 \rvert}}

\newcommand{\ol}{\overline}
\DeclareMathOperator\jdt{jdt}

\newcommand{\cross}[2]{ \draw[thick] (#1,#2)--(#1-1,#2-1);
     \draw[thick] (#1-1,#2)--(#1,#2-1);}
\renewcommand{\cir}[2]{ \draw[fill] (#1-0.5,#2-0.5) circle (.2);}
\newcommand{\perm}[1]{
 \foreach \y [count=\x] in #1 {
     \cross{\x}{\y}
}}
\newcommand{\permc}[1]{
 \foreach \y [count=\x] in #1 {
     \cir{\x}{\y}
}}

\newcommand{\crosscolor}[3]{ \draw[thick,#3] (#1,#2)--(#1-1,#2-1);
     \draw[thick,#3] (#1-1,#2)--(#1,#2-1);}
\newcommand{\circolor}[3]{ \draw[fill,#3] (#1-0.5,#2-0.5) circle (.2);}

\newcommand{\permcolor}[1]{
 \foreach \y [count=\x] in #1 {
     \ifthenelse{\y>0}{\crosscolor{\x}{\y}{blue}}{\ifthenelse{\y<0}{\crosscolor{\x}{-\y}{red}}{}}
}}
\newcommand{\permccolor}[1]{
 \foreach \y [count=\x] in #1 {
     \ifthenelse{\y>0}{\circolor{\x}{\y}{blue}}{\ifthenelse{\y<0}{\circolor{\x}{-\y}{red}}{}}
}}

\newcommand{\typeA}[1]{
\coordinate (a) at (1,0);
\coordinate (b) at (.5,.6);
    \foreach \row in {0,...,#1} {
        \draw ($\row*(a)$) -- ($#1*(b)+\row*(a)-\row*(b)$);
        \draw ($\row*(a)$) -- ($\row*(b)$);
\foreach \x in {0,...,#1}{
	\foreach \y in {\x,...,#1}{
		\filldraw[fill=white] ($\x*(a)+#1*(b)-\y*(b)$) circle (.1);
}}}}

\newcommand{\typeAtilted}[1]{
\coordinate (a) at (1,1);
\coordinate (b) at (0,1);
    \foreach \row in {0,...,#1} {
        \draw ($\row*(a)$) -- ($#1*(b)+\row*(a)-\row*(b)$);
        \draw ($\row*(a)$) -- ($\row*(b)$);
\foreach \x in {0,...,#1}{
	\foreach \y in {\x,...,#1}{
		\filldraw[fill=white] ($\x*(a)+#1*(b)-\y*(b)$) circle (.1);
}}}}

\newcommand{\typeB}[1]{
\coordinate (a) at (1,0);
\coordinate (b) at (.5,.6);
    \foreach \row in {0,...,#1} {
        \draw ($\row*(a)$) -- ($\row*(a)+2*#1*(b)-2*\row*(b)$);
        \draw ($\row*(b)$)--($\row*(a)$);
        \draw ($2*#1*(b)-\row*(b)$)--($\row*(a)+2*#1*(b)-2*\row*(b)$);
        }
    \foreach \x in {0,...,#1}{
	\foreach \y in {\x,...,#1}{
		\filldraw[fill=white] ($\x*(a)+#1*(b)-\y*(b)$) circle (.1);
        }
    }
    \foreach \y in {0,...,#1}
        \foreach \x in {0,...,\y}{
		\filldraw[fill=white] ($#1*(b)+\y*(b)+\x*(a)-2*\x*(b)$) circle (.1);
        }
}

\newcommand{\el}[2]{\filldraw ($#1*(a)-(a)+#2*(b)-#1*(b)$) circle (.1);}
\newcommand{\an}[2]{\el{#1}{#2} \draw[orange] ($#1*(a)-(a)+#2*(b)-#1*(b)$) circle (.16);}
\newcommand{\andarkgreen}[2]{\el{#1}{#2} \draw[darkgreen] ($#1*(a)-(a)+#2*(b)-#1*(b)$) circle (.16);}

\newcommand\start{circle(.05)}
\newcommand\up{-- ++(b) \start}
\newcommand\dn{-- ++($(a)-(b)$) \start}

\newcommand\match[2]{\draw (90+18-#1*36:1.4)--(90+18-#2*36:1.4);}
\newcommand\matc[2]{\draw(v#1)--(v#2);}
\definecolor{darkgreen}{rgb}{0,0.6,0}
\definecolor{darkblue}{rgb}{0,0,0.7}
\definecolor{darkred}{rgb}{0.7,0,0}
\definecolor{darkbrown}{rgb}{0.5,0.3,0}

\def\dgr{\textcolor{green!50!black}}

\def\rd{\textcolor{red!70}}

\def\ora{\textcolor{orange}}
\def\vio{\textcolor{violet}}

\def\N{-- ++(0,1)}
\def\E{-- ++(1,0)}
\DeclareMathOperator{\Pro}{Pro}
\DeclareMathOperator{\Rot}{Rot}
\DeclareMathOperator{\SYT}{SYT}
\DeclareMathOperator\AST{AST}
\DeclareMathOperator\RSK{RSK}
\newcommand\RSKD{\widehat{\RSK}}
\newcommand\ten{10}

\title{Rowmotion on $321$-avoiding permutations}

\author{Ben Adenbaum and Sergi Elizalde\thanks{Department of Mathematics, Dartmouth College, Hanover, NH 03755. \texttt{benjamin.m.adenbaum.gr@dartmouth.edu}, \texttt{sergi.elizalde@dartmouth.edu}}
}

\date{}

\begin{document}

\maketitle

\begin{abstract}
We give a natural definition of rowmotion for $321$-avoiding permutations, by translating, through bijections involving Dyck paths and the Lalanne--Kreweras involution, the analogous notion for antichains of the positive root poset of type $A$. We prove that some permutation statistics, such as the number of fixed points, are homomesic under rowmotion, meaning that they have a constant average over its orbits. 

Our setting also provides a more natural description of the celebrated Armstrong--Stump--Thomas equivariant bijection between antichains and non-crossing matchings in types $A$ and $B$, by showing that it is equivalent to the Robinson--Schensted--Knuth correspondence on $321$-avoiding permutations permutations. 
\end{abstract}



\section{Introduction}\label{sec:intro}

The goal of this work is two-fold. On the one hand, we initiate the study of pattern-avoiding permutations through the lens of dynamical algebraic combinatorics. On the other hand, we use these permutations to show that a celebrated bijection of Armstrong, Stump and Thomas between certain antichains and non-crossing matchings has a more natural description in terms of the well-known Robinson--Schensted--Knuth correspondence.

Let $\S_n$ denote the set of permutations of $\{1,2,\dots,n\}$.
We say that $\pi\in\S_n$ is {\rm $321$-avoiding} if there do not exist $i<j<k$ such that $\pi(i)>\pi(j)>\pi(k)$. 
Let $\S_n(321)$ denote the set of $321$-avoding permutations in $\S_n$.

Rowmotion is an operation defined on antichains of a finite poset, or equivalently, on its order ideals; see 
Section~\ref{rowmotion} for definitions. 
Historically, rowmotion was first described by Brouwer and Schrijver~\cite{BS74}, and then again by Cameron and Fon-der-Flaas~\cite{CF95} as a composition of certain involutions called toggles. The name of rowmotion comes from the work of Striker and Williams~\cite{SW12} where, for certain posets, rowmotion is described as a composition of toggles along the rows, and it is shown to be related to another operation called promotion.

We will restrict our attention to the poset of positive roots for the type $A$ 
root system. Antichains of this poset are in bijection with Dyck paths and with $321$-avoiding permutations. This will allow us to define a natural rowmotion operation on $\S_n(321)$.

When studying rowmotion, it is common to look for statistics that exhibit a property called {\em homomesy} \cite{PR15}. Given a set $S$ and a bijection $\tau:S\to S$ so that each orbit of the action of $\tau$ on $S$ has finite order, we say that a statistic on $S$ is {\em homomesic} under this action if its average on each orbit is constant. 
More specifically, the statistic is said to be {\em $c$-mesic} if its average over each orbit is $c$. 

We will prove that several statistics on $321$-avoiding permutations, including the number of fixed points (Theorem~\ref{thm:fp}), are homomesic under rowmotion. We will also show (Theorem~\ref{SignRowmotionTheorem}) that the sign statistic on $\S_n(321)$ is preserved by rowmotion when $n$ is odd, and it alternates when $n$ is even. 

In the second part of the paper, we use the viewpoint of $321$-avoiding permutations to shed new light into a celebrated bijection 
of Armstrong, Stump and Thomas \cite{AST13} between antichains in root posets of finite Weyl groups (also known as nonnesting partitions) and noncrossing partitions. This is an equivariant bijection in the sense that it translates rowmotion on antichains into an operation  called Kreweras complementation on noncrossing partitions which, in the classical types, is equivalent to rotation of noncrossing matchings. We will show that, in the case of types $A$ and $B$, the Armstrong--Stump--Thomas (AST) bijection 
has a simple interpretation in terms of the Robinson--Schensted--Knuth (RSK) correspondence applied to $321$-avoiding permutations.

The remainder of the paper is structured as follows. In Section~\ref{background} we review some properties of $321$-avoiding permutations and RSK, in addition to some background on rowmotion. In Section~\ref{defs} we define rowmotion on $321$-avoiding permutations, and use this definition to provide a simplified proof of a result of Hopkins and Joseph \cite[Thm.\ 6.2]{HJ20}
enumerating certain antichains. In Section~\ref{homomesies} we prove homomesy results for rowmotion on $321$-avoiding permutations, and we study the behavior of the sign statistic under rowmotion. Finally, in Section~\ref{sec:AST}, we provide an alternative description of the AST bijection from~\cite{AST13} in types $A$ and $B$ in terms of the RSK correspondence. This description allows us to derive some properties of AST from well-known properties of RSK, and to answer another question of Hopkins and Joseph \cite[Remark 6.7]{HJP}.

\section{Background}\label{background}

In this section we review some notions about Dyck paths, noncrossing matchings, and the RSK correspondence, in particular as it applies to $321$-avoiding permutations. We also provide a basic overview of rowmotion.

\subsection{Dyck paths, permutations, and noncrossing matchings}\label{sec:Dyck}

Let $\D_n$ be the set of words over $\{\uu,\dd\}$ consisting of $n$ $\uu$s and $n$ $\dd$s, and satisfying that every prefix contains at least as many $\uu$s as $\dd$s. Elements of $\D_n$ are called Dyck paths, and they will be drawn in three different ways as lattice paths in $\mathbb{Z}^2$ starting at the origin. Replacing $\uu$ and $\dd$ with $(0,1)$ and $(1,0)$ (respectively, $(1,0)$ and $(0,1)$), we obtain paths that stay weakly above (respectively, below) the diagonal $y=x$. We denote these by $\Dul_n$ (respectively, $\Dlr_n$). The sets $\Dul_n$ and $\Dlr_n$ are in bijection with each other, by simply reflecting along the diagonal $y=x$. 
The third way to draw Dyck paths that we will use is when $\uu$ and $\dd$ are replaced with $(1,1)$ and $(1,-1)$, respectively, so that the resulting path ends at $(2n,0)$ and never goes below the $x$-axis. 
In all cases, a pair of consecutive steps $\uu\dd$ is called a {\em peak}, and a pair $\dd\uu$ is called a {\em valley}. 

Interpreting $\uu$ and $\dd$ steps of $D\in\D_n$ as opening and closing parentheses, respectively, and matching them in the usual way, a pair of matched steps will be called a \emph{tunnel}, following \cite{EP1}. When $D$ is drawn as a path with $(1,1)$ and $(1,-1)$ and steps, a tunnel corresponds to a horizontal segment between two lattice points of $D$ that intersects $D$ only at these two points, and otherwise stays always below $D$. A tunnel is a \emph{centered tunnel} if it is centered with respect to the vertical line $x=n$ through the middle of $D$, a right (left) tunnel if its endpoints are strictly to the right (left) of this line, and a right-across (left-across) tunnel if its endpoints lie on opposite sides of this line and its center is to the right (left) of this line.

Several bijections between $321$-avoiding permutations and Dyck paths are known. 
We can represent $\pi \in \S_n$ as an $n\times n$ array with crosses in squares $(i,\pi(i))$ for $1\le i\le n$; we call this the {\em array} of $\pi$. Rows and columns are indexed using cartesian coordinates, so that $(i,j)$ denotes the cell in the $i$th column from the left and $j$th row from the bottom.
We say that $(i,\pi(i))$ is a {\em fixed point} (respectively {\em excedance}, {\em weak excedance}, {\em deficiency}, {\em weak deficiency}) if
$\pi(i)=i$ (respectively $\pi(i)>i$, $\pi(i)\ge i$, $\pi(i)<i$, $\pi(i)\le i$). 

For $\pi \in \S_n(321)$, let $E_p(\pi)\in\Dul_n$ be the path whose peaks occur at the weak excedances of $\pi$, let $E_v(\pi)\in\Dul_n $ be the path whose valleys occur at the excedances of $\pi$, and let $D_v(\pi)\in\Dlr_n$ be the path whose valleys occur at the weak deficiencies of $\pi$. See Figure~\ref{fig:321Dyck} for an example. These bijections appear in~\cite{Eli}. The bijection that maps $E_p(\pi)$ to $D_v(\pi)$ is known as the {\em Lalanne--Kreweras involution} on Dyck paths \cite{KRE70,LAL92}, which we denote by 
\begin{equation}\label{def:LK}\LK=D_v\circ E_p^{-1}.\end{equation} 

\begin{figure}[h]
\centering
\begin{tikzpicture}[scale=0.5]
 \draw (0,0) grid (9,9);
 \draw[dotted] (0,0)--(9,9);
 \perm{{2,4,1,3,5,8,9,6,7}}
 \draw[green!70!black,dotted,ultra thick] (0,0) \N\E\N\N\E\N\N\N\N\E\E\E\E\N\E\N\E\E;
  \draw[green!70!black] (.9,5.5) node {$E_v(\pi)$};
 \draw[red,ultra thick] (0,0) \N\N\E\N\N\E\E\E\N\E\N\N\N\E\N\E\E\E;
 \draw[red] (3.9,8.5) node {$E_p(\pi)$};
 \draw[blue,ultra thick] (0,0) \E\E\N\E\N\N\E\E\E\E\N\N\N\E\N\E\N\N;
   \draw[blue] (6,2.5) node {$D_v(\pi)$};
\end{tikzpicture}
\caption{The paths $E_p(\pi)$, $E_v(\pi)$ and $D_v(\pi)$ for the permutation $\pi=241358967$.}
\label{fig:321Dyck}
\end{figure}
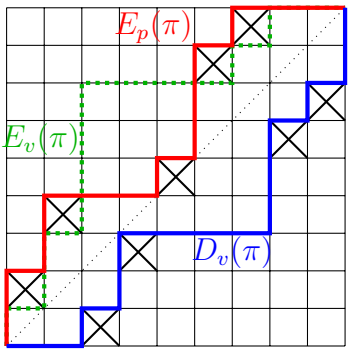

Let $\NC_n$ denote the set of {\em noncrossing matchings} of $\{1,2,\dots,2n\}$, i.e., perfect matchings with the property that there do not exist $i<j<k<\ell$ such that $i$ is matched with $k$ and $j$ is matched with $\ell$. We will draw the points $1,2,\dots,2n$ equally spaced around a circle in clockwise order, so that $1$ and $2n$ are near the top, and vertices $i$ and $2n+1-i$ are symmetrically placed with respect to a vertical axis. 
We draw a line segment, called an arc, connecting each pair of matched points. 

Define a clockwise rotation map $\Rot:\NC_n\to \NC_n$, which takes each arc $(i,j)$ to the arc $(i+1,j+1)$, with addition modulo $2n$. A matching is said to be \emph{centrally symmetric} if it is fixed under $180^\circ$ rotation.  There is a straightforward bijection between Dyck paths and noncrossing matchings, see Figure~\ref{fig:Dyck-matching} for an example.

\begin{definition}\label{def:Match}
Let $\Match:\D_n\to\NC_n$ be the bijection defined as follows.
Given $D\in\D_n$, the points $x$ and $y$ are matched in $\Match(D)$ if the steps of $D$ in positions $x$ and $y$ form a tunnel.
\end{definition}

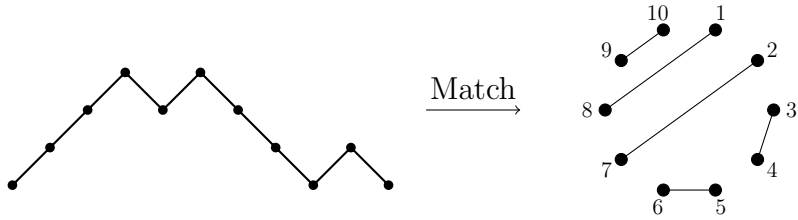
\begin{figure}[h]
\centering
\begin{tikzpicture}
\coordinate (a) at (1,0);
\coordinate (b) at (.5,.5);
    \filldraw[thick] (0,1) \start\up\up\up\dn\up\dn\dn\dn\up\dn;
\draw[-to ] (5.5, 2) --node[above]{$\Match$} (6.75,2);
\begin{scope}[shift={(9,2)},scale=.8]
\foreach \i in {1,...,10} {
\filldraw (90+18-\i*36:1.4) coordinate (v\i) circle (.1);
\draw (90+18-\i*36:1.7) node[scale=.7] {\i};
}
\matc{1}{8}\matc{2}{7}\matc{3}{4}\matc{5}{6}\matc{9}{10}
\end{scope}
\end{tikzpicture}
\caption{A Dyck path with the corresponding noncrossing matching.}
\label{fig:Dyck-matching}
\end{figure}

\subsection{RSK and $321$-avoiding permutations}\label{sec:RSK321Basics}
The RSK correspondence is a bijection between permutations and pairs of standard Young tableaux of the same shape. We refer the reader to~\cite[Sec.\ 7.11]{EC2} for definitions. Given a permutation $\pi\in\S_n$, we denote its image by $\RSK(\pi)=(P,Q)$, where $P$ and $Q$ are standard Young tableaux of the same shape, called the {\em insertion} and the {\em recording} tableaux of $\pi$, respectively. A well-known property of this correspondence \cite[Thm.\ 2]{SCH61} is that the number of rows of $P$ (equivalently, of $Q$) equals the length of the longest decreasing subsequence of $\pi$. In, particular, $\pi$ is $321$-avoiding if and only if $P$ and $Q$ have at most two rows. This property is used to define the following map from $\S_n(321)$ to $\D_n$, 
used in~\cite{EP1}.

\begin{definition}\label{def:RSKD}
Let $\pi\in \S_n(321)$, and suppose that $\RSK(\pi)=(P,Q)$. Define a Dyck path $\RSKD(\pi)$ 
as follows. For $1\le i\le n$, let the $i$th step be a $\uu$ if $i$ is in the top row of $P$, and a $\dd$ otherwise; let the $(2n+1-i)$th step be a $\dd$ if $i$ is in the top row of $Q$, and a $\uu$ otherwise.
\end{definition}

\begin{figure}[h]
\centering
\begin{tikzpicture}
\draw (0,0) node[left] {$41235$};
\draw[->] (.5,0)--node[above]{$\RSK$} (1,0);
\draw(4,0) node {$\bigg (\vio{\young(1235,4)},\quad\dgr{\young(1345,2)}\bigg)$};
\draw[->] (7,0)-- (7.5,0);
\coordinate (a) at (1,0);
\coordinate (b) at (.5,.5);
    \filldraw[thick,violet] (8,-.5) \start\up\up\up\dn\up;
    \filldraw[thick,green!50!black] (10.5,1) \dn\dn\dn\up\dn;
\draw[->] (-.5,.5) to [bend left =15] node[above]{$\RSKD$} (8,.5);
\end{tikzpicture}
\caption{The maps $\RSK$ and $\RSKD$ applied to the permutation $41235$.}
\label{fig:RSK}
\end{figure}
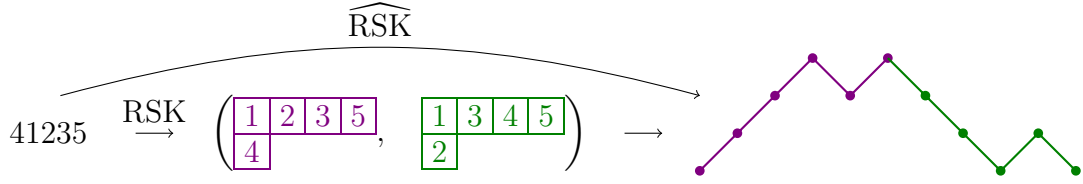

See Figure~\ref{fig:RSK} for an example of the these maps. It is clear that if $\pi$ is $321$-avoiding, then so is its inverse $\pi^{-1}$. An additional consequence of \cite[Thm.\ 2]{SCH61} is that if $\RSK(\pi)=(P,Q)$, then $\RSK(\pi^{-1})=(Q,P)$. This implies the following property of $\RSKD$, which will be used later on.

\begin{lemma}\label{lem:RSKDsym}
For $\pi\in\S_n(321)$, the path $\RSKD(\pi^{-1})$ is the reflection of the path $\RSKD(\pi)$, obtained by reversing the word and swapping $\emph{\uu}$s and $\emph{\dd}$s.
\end{lemma}
Consequently, the Dyck path $\RSKD(\pi)$ is symmetric if and only if $\pi$ is an involution, i.e., $\pi=\pi^{-1}$.

\subsection{Promotion and evacuation}\label{sec:promotion}

For a partition $\lambda$, let $\SYT(\lambda)$ be the set of standard Young tableaux of shape $\lambda$.
A straightforward bijection $\Tab:\D_n\to\SYT(n,n)$ can be described as follows. For given $D\in\D_n$, the entries in the top (resp.\ bottom) row of $\Tab(D)$ are the indices of the $\uu$ (resp.\ $\dd$) steps of $D$. For example, for the Dyck path in Figure~\ref{fig:Dyck-matching}, the corresponding tableau is $$\young(12359,4678\ten).$$

We begin by recalling the definitions of promotion and evacuation on standard Young tableaux using the Bender--Knuth involutions, which we refer to as toggles. 
\begin{definition}
Let $\lambda$ be a partition of $n$. For $1\le i \le n-1$, define the toggle $t_i:\SYT(\lambda)\to\SYT(\lambda)$ 
by letting $t_i(T)$ be standard Young tableau obtained from $T\in\SYT(\lambda)$ by switching the labels $i$ and $i+1$ if they are in different rows or columns. Note that $t_i$ is an involution. Define promotion and evacuation on $\SYT(\lambda)$ to be the maps $\Pro = t_{n-1}t_{n-2}\cdots t_{2}t_1$ and $\evac = (t_1)(t_2 t_1)\cdots(t_{n-1}t_{n-2}\cdots t_2 t_1)$ respectively.
\end{definition}

\begin{example}\label{ex:Pro}
Here is the computation of $\Pro(T)$ as a composition of toggles, for a tableau $T\in\SYT(5,5)$: 
\newcommand\rone{\rd{1}}
\newcommand\rtwo{\rd{2}}
\newcommand\rthree{\rd{3}}
\newcommand\rfour{\rd{4}}
\newcommand\rfive{\rd{5}}
\newcommand\rsix{\rd{6}}
\newcommand\rseven{\rd{7}}
\newcommand\reight{\rd{8}}
\newcommand\rnine{\rd{9}}
\newcommand\rten{\rd{10}}
\begin{multline*}
T=\young(\rone\rtwo478,3569\ten)\stackrel{t_1}{\to} \young(1\rtwo478,\rthree569\ten)\stackrel{t_2}{\to} \young(1\rthree\rfour78,2569\ten)\stackrel{t_3}{\to} \young(13\rfour78,2\rfive69\ten)\stackrel{t_4}{\to}\young(13\rfive78,24\rsix9\ten)\\
\stackrel{t_5}{\to}\young(135\rseven8,24\rsix9\ten)\stackrel{t_6}{\to}\young(1356\reight,24\rseven9\ten)\stackrel{t_7}{\to}\young(13567,24\reight\rnine\ten)\stackrel{t_8}{\to}\young(13567,248\rnine\rten)\stackrel{t_9}{\to}\young(13567,2489\ten)=\Pro(T)
\end{multline*}
\end{example}

We will use an alternative description of evacuation on standard Young tableaux, which relies on the following rectification process. 
For partitions $\lambda, \mu$ with $\lambda_i\ge \mu_i$ for all $i$, 
the skew shape $\lambda/ \mu$ consists of the boxes in the Young diagram of $\lambda$ which are not in that of $\mu$. An inner corner of $\lambda/\mu$ is a box that is not in $\lambda/\mu$ but both the boxes below and to its right are. 
A standard Young tableau of shape $\lambda/\mu$ is a filling of this shape with distinct entries $1,2,\dots,|\lambda/\mu|$ so that rows and columns are increasing. Denote the set of standard Young tableau of shape $\lambda/\mu$ by $\SYT(\lambda/\mu)$.

Given a tableau in $T\in\SYT(\lambda/\mu)$, let $\jdt(T)$ be the standard Young tableau of straight shape obtained from $T$ via the following procedure, called {\em rectification}. Beginning with an inner corner $c$ of $\lambda/\mu$, swap $c$ with the filled box below or to the right of $c$ with the smallest entry. If, after this swap, $c$ still has filled boxes below or to the right, keep applying this procedure until this is no longer the case. Then repeat with another inner corner of the resulting shape, until there are no more inner corners left, at which point the resulting tableau $\jdt(T)$ has straight shape.

\begin{lemma}[{\cite[Remark 2.4.24]{B94}}]\label{lem:EvacDescription}
For $T\in\SYT(\lambda)$ with $n$ boxes, 
let $\widetilde{T}$ be the skew tableau obtained by rotating $T$ by $180^\circ$ and replacing each entry $i$ with $n+1-i$. Then $$\evac(T)=\jdt(\widetilde{T}).$$ 
\end{lemma}

Next we define a rotation operation for Dyck paths, which we also call promotion. The reason for this name is that, as observed by White \cite[Sec.\ 8]{RHO10}, applying this operation to a path $D\in\D_n$ is equivalent to applying $\Pro$ to the tableau $\Tab(D)\in\SYT(n,n)$. 
\begin{definition}
Promotion of Dyck paths is the map $\Pro:\D_n\to \D_n$ defined as follows. Given a path $D\in\D_n$, consider its first-return decomposition as $D=\uu A\dd B$ where $A,B$ are Dyck paths, and let $\Pro(D)=A\uu B\dd$. 
\end{definition}
\begin{example}
The promotion of the path $D=\uu\uu\dd\uu\dd\dd\uu\uu\dd\dd$ is
$\Pro(D)=\uu\dd\uu\dd\uu\uu\uu\dd\dd\dd$, as shown in Figure~\ref{fig:Rot}. Note that $\Tab(D)=T$ and $\Tab(\Pro(D))=\Pro(T)$, with $T$ as in Example~\ref{ex:Pro}.
\end{example}

\begin{figure}[htb]
\centering
\begin{tikzpicture}
    \coordinate (a) at (1,0);
\coordinate (b) at (.5,.5);
\filldraw[thick] (0,0)\start\up;
\filldraw[thick](2.5,.5)\start\dn;
\filldraw[thick, color = violet] (.5,.5)\start\up\dn\up\dn;
\filldraw[thick, color = green] (3,0)\start\up\up\dn\dn;
\draw[->] (5.5,.3)--node[above]{$\Pro$} (6,.3);
\begin{scope}[shift={(6.5,0)}]
    \coordinate (a) at (1,0);
\coordinate (b) at (.5,.5);
\filldraw[thick] (2,0)\start\up;
\filldraw[thick](4.5,.5)\start\dn;
\filldraw[thick, color = green] (2.5,.5)\start\up\up\dn\dn;
\filldraw[thick, color = violet] (0,0)\start\up\dn\up\dn;
\end{scope}
\end{tikzpicture}
\caption{Promotion of Dyck paths. The paths $A$ and $B$ are colored in violet and green, respectively.}
\label{fig:Rot}
\end{figure}
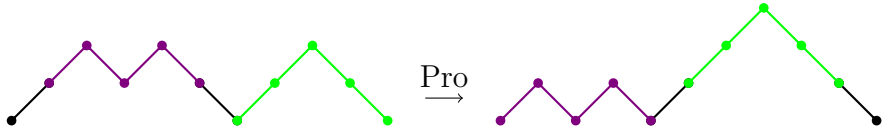

\subsection{Rowmotion}\label{rowmotion}

Let $\P$ be a finite poset. An {\em antichain} of $\P$ is a subset of pairwise incomparable elements. Denote by 
$\cA(P)$ the set of antichains of $\P$.
An {\em order ideal} of $\P$ is a subset $I$ with the property that if $x\in I$ and $y\le x$, then $y\in I$.
Similarly, an {\em order filter} of $\P$ is a subset $F$ with the property that if $x\in F$ and $x\le y$, then $y\in F$. Let $\I(\P)$ and $\F(\P)$ denote the sets of order ideals and order filters of $\P$, respectively.
The complementation map $\Theta$, defined on subsets $S$ of $\P$ by $\Theta(S) = \P\setminus S$, restricts to a bijection between $\I(P)$ and $\F(P)$.

Following~\cite{HJ20}, define the up-transfer map $\Delta:\I(\P)\to \cA(\P)$ by letting $\Delta(I)$ be the set of maximal elements of $I\in\I(P)$, and the down-transfer map $\nabla: \F(\P)\to \cA(\P)$ by letting $\nabla(F)$ be the set of minimal elements of $F\in\F(P)$. The inverses of these bijections are given by 
\begin{align*}\Delta^{-1}(A) &= \{x\in \P : x\le y\text{ for some } y\in A\},\\
\nabla^{-1}(A) & = \{x\in \P : x\ge y\text{ for some } y\in A\},\end{align*} 
for $A\in\cA(P)$. 

Antichain rowmotion is the map $\rA:\cA(\P)\to \cA(\P)$ defined as the composition $\rA = \nabla\circ\Theta\circ\Delta^{-1}$. In words,
for $A\in\cA(\P)$, the antichain $\rA(A)$ consists of the minimal elements of the complement of the order ideal generated by $A$.
Similarly, order ideal rowmotion, denoted by $\rI:\I(\P)\to \I(\P)$, is defined as $\rI = \Delta^{-1}\circ\nabla\circ\Theta$.
Finally, order filter rowmmotion, denoted by $\rF:\F(\P)\to\F(\P)$, is defined as $\rF=\Theta\circ\Delta^{-1}\circ \nabla$. See Figure~\ref{fig:waltz} for examples.

\begin{figure}[h]
    \centering
    \begin{tikzpicture}[scale=.75]
\typeA{4}    
\an{1}{1}\an{2}{4}\an{3}{5}
\draw (2,3.5) node{$\cA(\P)$};
\draw[->] (4.5,1)--node[above,align=center]{$\Delta^{-1}$}(5,1);
\draw[->] (2,-.5)--node[right]{$\rA$}(2,-1.3);
\begin{scope}[shift={(5.5,0)}]
\typeA{4}    
\an{1}{1}\an{2}{4}\an{3}{5}
\el{1}{1}\el{2}{2}\el{3}{3}\el{4}{4}\el{5}{5}
\el{2}{3}\el{3}{4}\el{4}{5}\el{2}{4}\el{3}{5}
\draw (2,3.5) node{$\I(\P)$};
\draw[->] (4.5,1)--node[above]{$\Theta$}(5,1);
\draw[->] (2,-.5)--node[right]{$\rI$}(2,-1.3);
\end{scope}

\begin{scope}[shift={(11,0)}]
\typeA{4}    
\el{1}{2}\el{1}{3}\el{1}{4}\el{1}{5}\el{2}{5}
\draw (2,3.5) node{$\F(\P)$};
\draw[->] (2,-.5)--node[right]{$\rF$}(2,-1.3);
\end{scope}

\begin{scope}[shift={(0,-4.3)}]
\typeA{4}    
\an{1}{2}\an{2}{5}
\draw[->] (4.5-5.5,1)--node[above,align=center]{$\nabla$}(5-5.5,1);
\draw[->] (4.5,1)--node[above,align=center]{$\Delta^{-1}$}(5,1);
\draw[->] (2,-.5)--node[right]{$\rA$}(2,-1.3);
\end{scope}

\begin{scope}[shift={(5.5,-4.3)}]
\typeA{4}    
\an{1}{2}\an{2}{5}
\el{1}{1}\el{2}{2}\el{3}{3}\el{4}{4}\el{5}{5}
\el{2}{3}\el{3}{4}\el{4}{5}\el{2}{4}\el{3}{5}
\draw[->] (4.5,1)--node[above]{$\Theta$}(5,1);
\draw[->] (2,-.5)--node[right]{$\rI$}(2,-1.3);
\end{scope}

\begin{scope}[shift={(11,-4.3)}]
\typeA{4}    
\el{1}{3}\el{1}{4}\el{1}{5}
\draw[->] (2,-.5)--node[right]{$\rF$}(2,-1.3);
\end{scope}

\begin{scope}[shift={(0,-8.6)}]
\typeA{4}    
\an{1}{3}
\draw[->] (4.5-5.5,1)--node[above,align=center]{$\nabla$}(5-5.5,1);
\draw[->] (4.5,1)--node[above,align=center]{$\Delta^{-1}$}(5,1);
\end{scope}

\begin{scope}[shift={(5.5,-8.6)}]
\typeA{4}    
\an{1}{3}
\el{1}{1}\el{2}{2}\el{3}{3}
\el{1}{2}\el{2}{3}
\draw[->] (4.5,1)--node[above]{$\Theta$}(5,1);
\draw (7,1) node {$\cdots$};
\end{scope}
\end{tikzpicture}
    \caption{An example of rowmotion as a composition of the bijections $\Delta$, $\Theta$, and $\nabla^{-1}$.}
    \label{fig:waltz}
\end{figure}
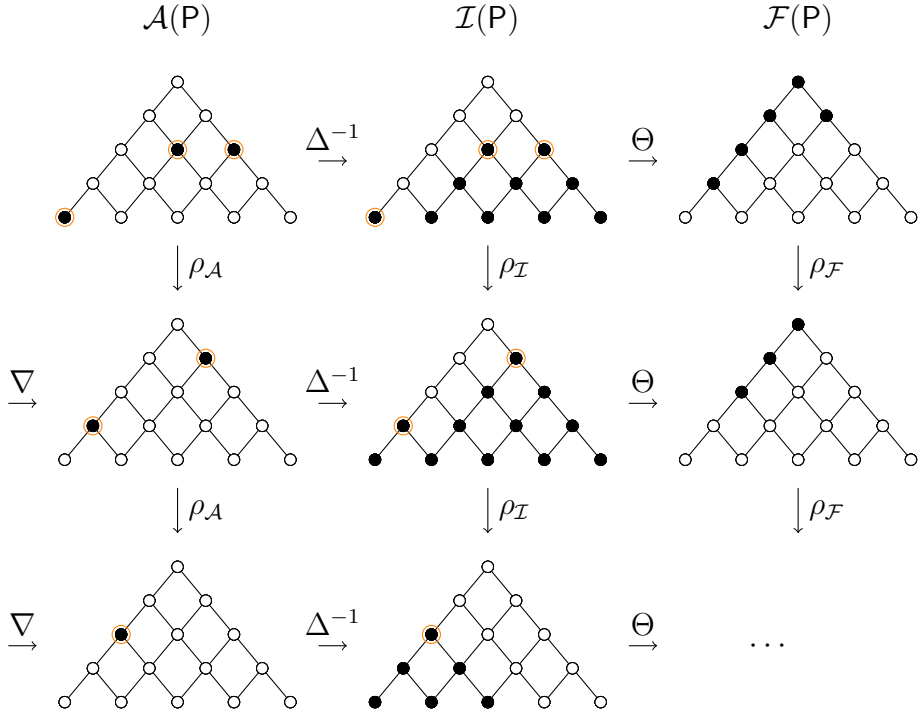

The poset of positive roots in type $A$, which we denote by $\A^{n-1}$, can be described as the set of intervals $\{[i,j]: 1\le i\le j\le n-1\}$ ordered by inclusion. It is a ranked poset, with rank function given by $\rk([i,j])=j-i$. The poset $\A^5$ appears in Figure~\ref{fig:waltz}.
The set $\cA(\A^{n-1})$ is in bijection with $\D_n$. One such bijection consists of mapping each path $D$ to the antichain $\ant(P)$ whose elements are at the valleys of $D$, as illustrated in Figure~\ref{fig:path}. 
This bijection allows us to define rowmotion on Dyck paths as $\rD=\ant^{-1}\circ \rA\circ \ant$. A second bijection between $\cA(\A^{n-1})$ and $\D_n$ consists of mapping the antichain $A$ to the path $\Path(A)$ whose peaks are at the elements of the antichain, as also illustrated in Figure~\ref{fig:path}. Note that $\rA =\ant\circ\Path$ and $\rD = \Path\circ \ant$.

Panyushev \cite{PAN04,PAN09} considered the map on $\cA(\A^{n-1})$ 
defined by mapping an antichain $\{[i_1,j_1],\dots [i_k, j_k]\}$ to $\{[i'_1,j'_1],\dots [i'_{n-k-1},j'_{n-k-1}]\}$ where $\{i'_1,\dots i'_{n-k-1}\}= \{1,2,\dots,n-1\}\setminus \{j_1,\dots, j_k\}$ and $\{j'_1,\dots j'_{n-k-1}\}= \{1,2,\dots,n-1\}\setminus \{i_1,\dots, i_k\}$. As noted in~\cite{HJ20}, this map equals the composition $\ant \circ \LK\circ \ant^{-1}$, so we will denote it by $\LKA$. Hopkins and Joseph show in \cite[Thm.\ 3.5]{HJ20} that a certain map known as \emph{antichain rowvacuation}, denoted in general by $\Rvac_\cA$, coincides with $\LKA$ in the case of the type $A$ root poset.

\begin{figure}[h]
\centering
\begin{tikzpicture}[scale=1]
\typeA{6}
\draw (0,0) node[below,scale=.8] {$[1,1]$};
\draw (1,0) node[below,scale=.8] {$[2,2]$};
\draw (2,0) node[below,scale=.8] {$[3,3]$};
\draw (6,0) node[below,scale=.8] {$[7,7]$};
\draw (.5,.6) node[left,scale=.8] {$[1,2]$};
\draw (1,1.2) node[left,scale=.8] {$[1,3]$};
\draw (3,3.6) node[left,scale=.8] {$[1,7]$};
\draw (1.5,.6) node[left,scale=.8] {$[2,3]$};
\andarkgreen{3}{4}\andarkgreen{6}{6} 
\an{1}{3}\an{4}{5}\an{7}{7}
\filldraw[red] ($-.5*(a)-(b)$)\start\up\up\up\up\dn\dn\dn\up\up\dn\dn\dn\up\up\dn\dn;
\end{tikzpicture}
    \caption{The antichain $A = \{[1,3],[4,5],[7,7]\}$ (orange) in the poset $\A^7$, the Dyck path $D=\Path(A)$ (red), and the antichain $\ant(D)=\{[3,4],[6,6]\}$ (green).}
    \label{fig:path}
\end{figure}
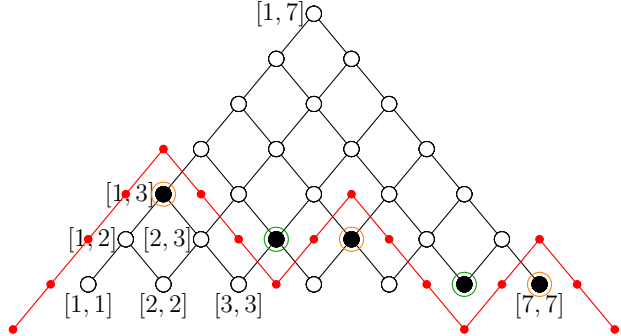

\section{A rowmotion operation on $321$-avoiding permutations}\label{defs}

To define rowmotion on $321$-avoiding permutations, first consider the following bijection to antichains of $\A^{n-1}$.

\begin{definition}\label{def:Exc}
Let $\Exc:\S_n(321)\to \cA(\A^{n-1})$ be the bijection where, for $\pi\in\S_n(321)$, we define
$$
\Exc(\pi)=\{[i,\pi(i)-1]\,:\,(i,\pi(i)) \text{ is an excedance of } \pi\}.
$$
\end{definition}

See Figure~\ref{fig:Exc} for an example. To see that this is a bijection, note that it can be written as $\Exc=\Path^{-1}\circ E_p=\ant\circ E_v$.

\begin{figure}[h]
    \centering
    \begin{tikzpicture}[scale=0.5]
 \draw (0,0) grid (9,9);
 \draw[dotted] (0,0)--(9,9);
 \perm{{2,4,1,3,5,8,9,6,7}}
 \crosscolor{1}{2}{orange}\crosscolor{2}{4}{orange}\crosscolor{6}{8}{orange}\crosscolor{7}{9}{orange}
\begin{scope}[shift={(.5,1.5)}]
\typeAtilted{7}
\an{1}{1}\an{2}{3}\an{6}{7}\an{7}{8}
\end{scope}
\end{tikzpicture}
\begin{tikzpicture}[scale=0.6]
\draw[->] (-1,2.3)--node[above]{$\Exc$} (-.2,2.3);
\typeA{7}
\an{1}{1}\an{2}{3}\an{6}{7}\an{7}{8}
\end{tikzpicture}  
    \caption{An example of the bijection $\Exc:\S_9(321)\to \cA(\A^{8})$, which maps the permutation $241358967$ to the antichain $\{[1,1],[2,3],[6,7],[7,8]\}$.}
    \label{fig:Exc}
\end{figure}
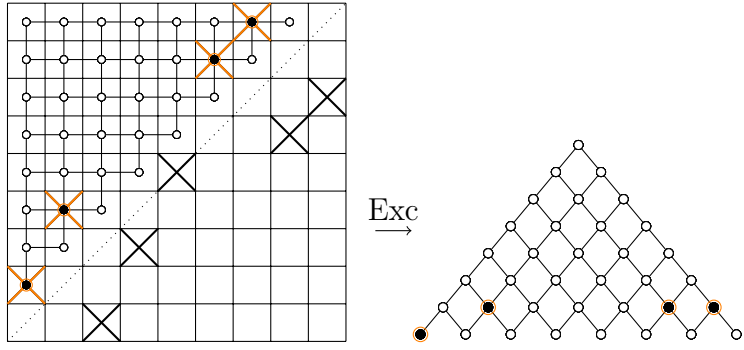

The map $\Exc$ provides the following natural way to translate rowmotion into an operation on $\S_n(321)$; see Figure~\ref{fig:rS_example} for an example.

\begin{definition}\label{def:rowmotion}
Rowmotion on $321$-avoiding permutations is the map $\rS:\S_n(321)\to \S_n(321)$ defined by $\rS=\Exc^{-1}\circ\rA\circ \Exc$.
\end{definition}

\begin{figure}[h]
\centering
\begin{tikzpicture}[scale=0.5]
 \draw (0,0) grid (9,9);
 \draw[dotted] (0,0)--(9,9);
 \perm{{2,4,1,3,5,8,9,6,7}}
 \draw (4.5,10) node {$\ora{24}135\ora{89}67$};
  \crosscolor{1}{2}{orange}\crosscolor{2}{4}{orange}\crosscolor{6}{8}{orange}\crosscolor{7}{9}{orange}
\draw[->] (9.5,4.5)-- node[above] {$\rS$} (10.5,4.5);
  \draw[->] (4.5,-.5)-- node[right] {$\Exc$} (4.5,-1.5);
  \draw[->] (9.5,4.5-9)-- node[above] {$\rA$} (10.5,4.5-9);
 \begin{scope}[shift={(0,-8)},scale=1.3]
 \typeA{7}
\an{1}{1}\an{2}{3}\an{6}{7}\an{7}{8}
  \end{scope} 
\begin{scope}[shift={(11,0)}]
  \draw (0,0) grid (9,9);
 \draw[dotted] (0,0)--(9,9);
 \perm{{3,1,2,5,6,9,4,7,8}}
 \draw (4.5,10) node {$\ora{3}12\ora{569}478$};
  \crosscolor{1}{3}{orange}\crosscolor{4}{5}{orange}\crosscolor{5}{6}{orange}\crosscolor{6}{9}{orange}
 \draw[->] (9.5,4.5)-- node[above] {$\rS$} (10.5,4.5);
   \draw[->] (9.5,4.5-9)-- node[above] {$\rA$} (10.5,4.5-9);
  \draw[->] (4.5,-.5)-- node[right] {$\Exc$} (4.5,-1.5);
\end{scope}
 \begin{scope}[shift={(11,-8)},scale=1.3]
 \typeA{7}
\an{1}{2}\an{4}{4}\an{5}{5}\an{6}{8}
  \end{scope}
\begin{scope}[shift={(22,0)}]
 \draw (0,0) grid (9,9);
 \draw[dotted] (0,0)--(9,9);
 \perm{{1,2,4,6,7,3,5,8,9}}
 \draw (4.5,10) node {$12\ora{467}3589$};
   \crosscolor{3}{4}{orange}\crosscolor{4}{6}{orange}\crosscolor{5}{7}{orange}
  \draw[->] (9.5,4.5)-- node[above] {$\rS$} (10.5,4.5);
   \draw[->] (4.5,-.5)-- node[right] {$\Exc$} (4.5,-1.5);
     \draw[->] (9.5,4.5-9)-- node[above] {$\rA$} (10.5,4.5-9);
\end{scope}
 \begin{scope}[shift={(22,-8)},scale=1.3]
 \typeA{7}
\an{3}{3}\an{4}{5}\an{5}{6}
  \end{scope}
\end{tikzpicture}
\caption{Two applications of rowmotion starting at $\pi=241358967\in\S_9(321)$, computed using Definition~\ref{def:rowmotion}.
}
\label{fig:rS_example}
\end{figure}
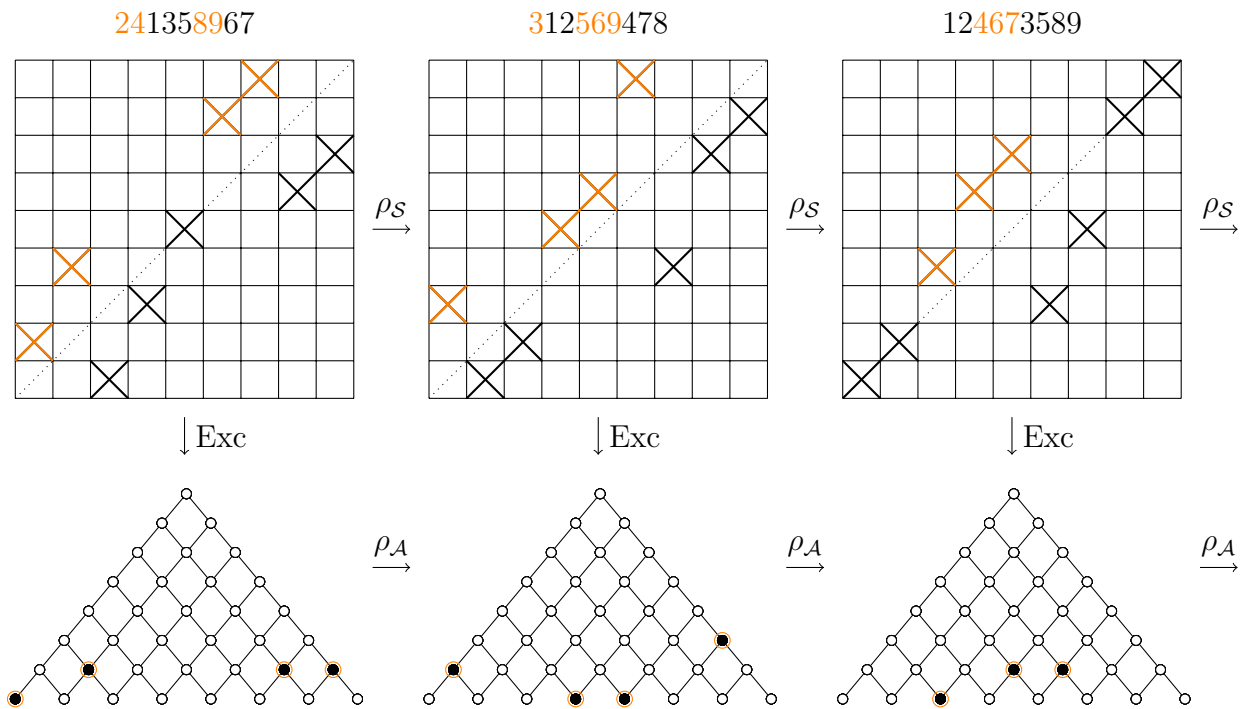

An equivalent description of $\rS$ can be given in terms of the above bijections to Dyck paths, as the composition
\begin{equation}\label{eq:rS_paths}
    \rS=E_v^{-1}\circ E_p.
\end{equation}
See Figure~\ref{fig:rS_example_paths} for examples of $\rS$ computed in this way.
Note that if $\pi\in\S_n(321)$ has upper and lower paths given by $P=E_p(\pi)$ and $Q=D_v(\pi)$, then $\sigma=\rS(\pi)$ has upper and lower paths given by $\rD(P)=E_p(\sigma)$ and $\rD^{-1}(Q)=D_v(\sigma)$. Equivalently, at the level of antichains, we have $\Exc(\sigma)=\rA(\Exc(\pi))$ and $\Exc(\sigma^{-1})=\rA^{-1}(\Exc(\pi^{-1}))$, noting that
$\Exc(\pi^{-1})$ is the antichain formed by the deficiencies of $\pi$. This follows from the fact that 
\begin{equation}\label{eq:Pan}\LKA\circ\rA=\rA^{-1}\circ\LKA,\end{equation}
which was proved by Panyushev \cite[Thm.\ 3.5]{PAN09}.

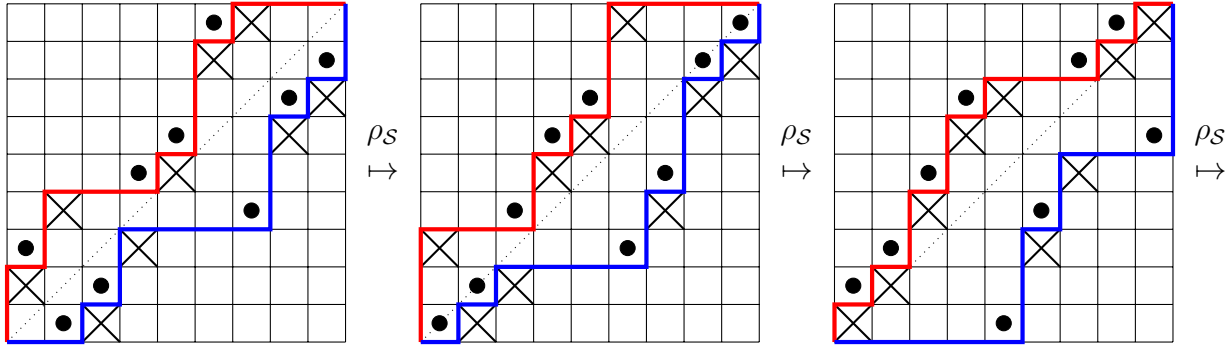
\begin{figure}[h]
\centering
\begin{tikzpicture}[scale=0.5]
 \draw (0,0) grid (9,9);
 \draw[dotted] (0,0)--(9,9);
 \perm{{2,4,1,3,5,8,9,6,7}}
 \draw[red,ultra thick] (0,0) \N\N\E\N\N\E\E\E\N\E\N\N\N\E\N\E\E\E;
 \draw[blue,ultra thick] (0,0) \E\E\N\E\N\N\E\E\E\E\N\N\N\E\N\E\N\N;
 \permc{{3,1,2,5,6,9,4,7,8}}
 \draw (10,4.5) node[label=$\rS$] {$\mapsto$};
\begin{scope}[shift={(11,0)}]
 \draw (0,0) grid (9,9);
 \draw[dotted] (0,0)--(9,9);
 \perm{{3,1,2,5,6,9,4,7,8}}
 \draw[red,ultra thick] (0,0) \N\N\N\E\E\E\N\N\E\N\E\N\N\N\E\E\E\E;
 \draw[blue,ultra thick] (0,0) \E\N\E\N\E\E\E\E\N\N\E\N\N\N\E\N\E\N;
 \permc{{1,2,4,6,7,3,5,8,9}}
 \draw (10,4.5) node[label=$\rS$] {$\mapsto$};
\end{scope}
\begin{scope}[shift={(22,0)}]
 \draw (0,0) grid (9,9);
 \draw[dotted] (0,0)--(9,9);
 \perm{{1,2,4,6,7,3,5,8,9}}
 \draw[red,ultra thick] (0,0) \N\E\N\E\N\N\E\N\N\E\N\E\E\E\N\E\N\E;
 \draw[blue,ultra thick] (0,0) \E\E\E\E\E\N\N\N\E\N\N\E\E\E\N\N\N\N;
 \permc{{2,3,5,7,1,4,8,9,6}}
 \draw (10,4.5) node[label=$\rS$] {$\mapsto$};
\end{scope}
\end{tikzpicture}
\caption{Rowmotion starting at $\pi=241358967\in\S_9(321)$, computed using the composition~\eqref{eq:rS_paths}. In the left diagram, the crosses represent $\pi$, the red path is $E_p(\pi)$, the blue path is $D_v(\pi)=\LK(E_p(\pi))$, and the dots represent $\rS(\pi)=E_v^{-1}(E_p(\pi))=312569478$.
}
\label{fig:rS_example_paths}
\end{figure}

A diagram of our bijections for permutations, paths and antichains, as well as their interactions, appears in Figure~\ref{fig:bijections}. 
We can translate the Lalanne--Kreweras involution for Dyck paths into an involution on $\S_n(321)$ by defining 
\begin{equation}\label{def:LKS}\LKS= E_p^{-1}\circ\LK\circ E_p.\end{equation}
Using that $\LKA=\ant \circ \LK\circ \ant^{-1}$, the maps $\LKS$ and $\LKA$ are related by
\begin{equation}\label{eq:LKS_LKA}\LKS= E_p^{-1}\circ\ant^{-1}\circ\LKA\circ\ant\circ E_p.\end{equation}

\begin{figure}[htb]
\centering
\begin{tikzpicture}[scale=1.7]
\draw node (S0) at (0,0) {$\S_n(321)$};
\draw node (S1) at (2,0) {$\S_n(321)$};
\draw node (D0) at (.7,1) {$\D_n$};
\draw node (D1) at (2.7,1) {$\D_n$};
\draw node (A0) at (1.4,2) {$\cA(\A^{n-1})$};
\draw node (A1) at (3.4,2) {$\cA(\A^{n-1})$};

\draw[gray] node (S0L) at (0,-3) {$\S_n(321)$};
\draw[gray] node (S1L) at (2,-3) {$\S_n(321)$};
\draw[gray] node (D0L) at (.7,-2) {$\D_n$};
\draw[gray] node (D1L) at (2.7,-2) {$\D_n$};
\draw[gray] node (A0L) at (1.4,-1) {$\cA(\A^{n-1})$};
\draw[gray] node (A1L) at (3.4,-1) {$\cA(\A^{n-1})$};
\draw[<->,dotted,very thick,green!50!black] (S0L) to [bend right=20] (S1);

\draw[<->,dashed,purple] (D0) --  (D0L);
\draw[<->,dashed,purple] (D1) -- node[right]{$\LK$} (D1L);
\draw[<->,dashed,purple] (S0) --  (S0L);
\draw[<->,dashed,purple] (S1) -- node[right]{$\LKS$} (S1L);
\draw[<->,dashed,purple] (A0) --  (A0L);
\draw[<->,dashed,purple] (A1) -- node[right]{$\LKA$} (A1L);
\draw[->] (S0)-- node[below left=-2mm]{$D_v$} (D0L);

\draw[->] (D0) -- node[above right]{$\rD$} (D1);
\draw[->] (S0) -- node[above]{$\rS$} (S1);
\draw[->] (A0) -- node[above]{$\rA$} (A1);
\draw[->] (S0) -- node[left]{$E_p$} (D0);
\draw[->] (S1) -- node[right]{$E_p$} (D1);
\draw[->] (S1) -- node[left]{$E_v$} (D0);
\draw[->] (D0) -- node[left]{$\ant$} (A0);
\draw[->] (D1) -- node[left]{$\ant$} (A1);
\draw[->,brown] (S1) -- node[below right]{$\Exc$} (A0);
\draw[->] (A0) -- node[above right]{$\Path$} (D1);

\draw[->,gray] (D1L) -- node[above]{$\rD$} (D0L);
\draw[->,gray] (S1L) -- node[above]{$\rS$} (S0L);
\draw[->,gray] (A1L) -- node[above]{$\rA$} (A0L);
\draw[->,gray] (S0L) -- node[left]{$E_p$} (D0L);
\draw[->,gray] (S1L) -- node[right]{$E_p$} (D1L);
\draw[->,gray] (D0L) -- node[left]{$\ant$} (A0L);
\draw[->,gray] (D1L) -- node[left]{$\ant$} (A1L);
\end{tikzpicture}

\caption{Diagram of the bijections $\rS$, $\rD$, $\rA$, $\ant$, $\Path$, $E_p$, $E_v$, $D_v$, $\LKS$, $\LKA$, and $\LK$. The vertical dashed arrows are the various versions of the Lalanne--Kreweras involution, and the dotted curved arrow is the map sending a permutation to its inverse.}
\label{fig:bijections}
\end{figure}
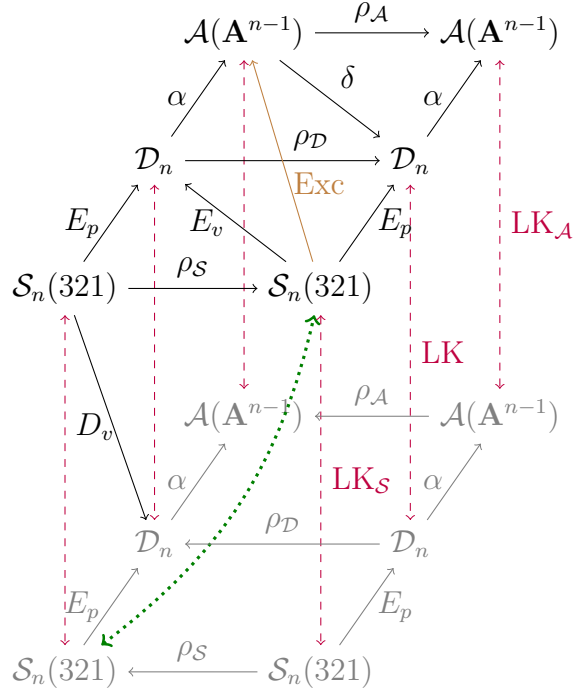

The map $\LKS$ is closely related to the inversion map on permutations, as the next lemma shows.

\begin{lemma}\label{LK_of_Perm}
	If $\pi\in \S_n(321)$ then $\pi^{-1} = \rho_{\S}(\LKS(\pi))=\LKS(\rho_{\S}^{-1}(\pi))$.
\end{lemma}
\begin{proof}
Using equations \eqref{eq:rS_paths}, \eqref{def:LKS} and~\eqref{def:LK}, in this order, we obtain $$\rho_{\S}(\LKS(\pi))=E_v^{-1}(E_p(\LKS(\pi))=E_v^{-1}(\LK( E_p(\pi)))=E_v^{-1}(D_v(\pi))=\pi^{-1},$$
where the last equality follows from the fact that $(i,j)$ is a deficiency of $\pi$ if and only if $(j,i)$ is an excendance of $\pi^{-1}$.
The same equality, with $\pi^{-1}$ playing the role of $\pi$, states that $\rho_{\S}(\LKS(\pi^{-1}))=\pi$. Consequently $\pi^{-1}=\LKS^{-1}(\rho_{\S}^{-1}(\pi)) =  \LKS(\rS^{-1}(\pi))$, using that $\LKS$ is an involution, which follows from the fact $\LK$ is an involution as well.
\end{proof}

Noting that  $\ant \circ E_p=\Exc \circ \rS$, we can rewrite Equation~\eqref{eq:LKS_LKA} as
\[ \rS\circ \LKS\circ \rS^{-1}= \Exc^{-1}\circ \LKA\circ \Exc. \]
Composing on the right with $\rS=\Exc^{-1}\circ\rA\circ\Exc$ (see Definition~\ref{def:rowmotion}), and using Lemma~\ref{LK_of_Perm}, we get
\begin{equation}\label{eq:LKS_LKA_EXC}
\LKS\circ \rS^{-1}=\rS\circ \LKS= \Exc^{-1}\circ \LKA\circ \rA \circ \Exc.
\end{equation}

In~\cite[Thm.\ 6.2]{HJ20}, Hopkins and Joseph enumerate antichains of $\A^{n-1}$ that are fixed by the involution $\LKA\circ\rA$. Next we show that the concept of rowmotion on $321$-avoiding permutations provides a simpler proof of this theorem, by reducing it to a classical result of Simion and Schmidt on the enumeration of pattern-avoiding involutions~\cite{SS85}.

\begin{theorem}[{\cite[Thm.\ 6.2]{HJ20}}]\label{thm:FixedOfLKA}
\[|\{A\in \cA(\mathbf{A}^{n-1}): \LKA(\rA(A))= A\}|=\binom{n}{\lfloor\frac{n}{2}\rfloor}\]
\end{theorem}

\begin{proof}
By Equation~\eqref{eq:LKS_LKA_EXC}, $\LKA\circ \rA$ is conjugate to $\LKS\circ \rS^{-1}$, so they have the same number of fixed points. By Lemma \ref{LK_of_Perm}, the number of fixed points of $\LKS\circ \rS^{-1}$ is the number of permutations $\pi\in\S_n(321)$ such that $\pi=\pi^{-1}$, i.e., the number of $321$-avoiding involutions. 
It is a classical result of Simion and Schmidt \cite{SS85} that this number equals the central binomial coefficient. 

\end{proof}

\section{Statistics and homomesies}\label{homomesies}

In this section we  show that certain statistics on $321$-avoiding permutations exhibit homomesy, as defined in Section~\ref{sec:intro}, under the action of $\rS$. 

\subsection{The number of fixed points}
The first statistic that we consider is the number of fixed points of a permutation $\pi$, denoted by $\fp(\pi)= |\{i : \pi(i) = i\}|$. It is interesting to note that, despite being a common statistic on permutations, the statistic on antichains that is obtained by translating it via the bijection $\Exc$ is far less natural, which explains why it has not been studied in the literature on antichain rowmotion.

Let $\exc(\pi)=|\{i:\pi(i)>i\}|$ and $\wexc(\pi)=|\{i:\pi(i)\ge i\}|$ denote the number of excedances and weak excedances of $\pi$, respectively. 

\begin{theorem}\label{thm:fp}
The statistic $\fp$ is 1-mesic under the action of $\rS$ on $\S_n(321)$.
\end{theorem}
 
\begin{proof}
Let $\pi\in\S_n(321)$, and let $\mathcal{O}$ be the orbit of $\pi$ under $\rS$. Then
$$\fp(\pi)=\wexc(\pi)-\exc(\pi)=n-\exc(\pi^{-1})-\exc(\pi)=n-\card{\Exc(\pi)}-\card{\Exc(\pi^{-1})}.$$
Summing over the orbit, $$\frac{1}{|\mathcal{O}|}\sum_{\pi\in \mathcal{O}}\fp(\pi) = n-\frac{1}{|\mathcal{O}|}\sum_{\pi\in \mathcal{O}}\card{\Exc(\pi)}-\frac{1}{|\mathcal{O}|}\sum_{\pi\in \mathcal{O}}\card{\Exc(\pi^{-1})}.$$
As noted above Equation~\eqref{eq:Pan}, applying $\rS$ to $\pi$ corresponds to applying $\rA$ to $\Exc(\pi)$ and $\rA^{-1}$ to $\Exc(\pi^{-1})$. Thus, by Lemma \ref{LK_of_Perm}, the sets $\{\Exc(\pi) : \pi\in \mathcal{O}\}$ and $\{\Exc(\pi^{-1}): \pi\in \mathcal{O}\}$ are complete orbits under $\rA$.
It is known~\cite{AST13} that the antichain cardinality statistic is $\frac{n-1}{2}$-mesic under the action of $\rA$ on $\cA(\A^{n-1})$.
Thus $\frac{1}{|\mathcal{O}|}\sum_{\pi\in \mathcal{O}}\card{\Exc(\pi)}=\frac{1}{|\mathcal{O}|}\sum_{\pi\in \mathcal{O}}\card{\Exc(\pi^{-1})}=\frac{n-1}{2}$, and we conclude that $\fp$ is $1$-mesic. 
\end{proof}

\subsection{The statistics $h_i$ and $\ell_i$}
Next we consider two families of statistics on $321$-avoiding permutations, and show that they are also homomesic under rowmotion. As a consequence, we obtain another proof of Theorem~\ref{thm:fp}.
The first family are the statistics $h_i$ introduced by  Hopkins and Joseph \cite{HJ20}. For $1\le i\le n-1$, they define $h_i$ on antichains $A\in\cA(\A^{n-1})$ as
$$h_i(A)=\sum_{j=1}^i \one_{[j,i]}(A)+\sum_{j=i}^{n-1}\one_{[i,j]}(A),$$ 
where $\one_{[i,j]}(A)$ is the indicator function that equals 1 if $[i,j]\in A$ and $0$ otherwise. 
For $\pi \in \S_n(321)$, we now define $h_i(\pi)$ to equal $h_i(\Exc(\pi))$. In terms of the array of $\pi\in\S_n(321)$, $h_i(\pi)$ counts the number of crosses of the form $(j,i+1)$ with $1\le j\le i$, plus the number of crosses of the form $(i,j)$ with $i+1\le j\le n$. Note that a cross in $(i,i+1)$ is counted twice. See the left of Figure~\ref{fig:hi_li} for a visualization. We can also write
\begin{equation}\label{def:hi} h_i(\pi)=\ind_{\pi^{-1}(i+1)<i+1} + \ind_{\pi(i)>i},\end{equation}
where $\ind_B $ is defined to be 1 if the statement $B$ is true and 0 otherwise. 

\begin{figure}[htb]
    \centering
\begin{tikzpicture}[scale=0.5]
\fill[gray!30!white] (0,3) rectangle (3,4);
\fill[gray!30!white] (2,3) rectangle (3,9);
\fill[gray!70!white] (2,3) rectangle (3,4);
 \draw (0,0) grid (9,9);
 \draw[dotted] (0,0)--(9,9);
 \perm{{3,1,4,2,6,7,9,5,8}}
 \draw (-1,3.5) node {$i+1$};
 \draw (2.5,9.5) node {$i$};
\end{tikzpicture}
\hspace{1cm}
\begin{tikzpicture}[scale=0.5]
\fill[gray!30!white] (0,2) rectangle (3,3);
\fill[gray!30!white] (2,3) rectangle (3,9);
 \draw (0,0) grid (9,9);
 \draw[dotted] (0,0)--(9,9);
 \perm{{3,1,4,2,6,7,9,5,8}}
 \draw (-1,2.5) node {$i$};
 \draw (2.5,9.5) node {$i$};
\end{tikzpicture}
\caption{Visualization of the statistics $h_i$ (left) and $\ell_i$ (right) on the permutation $\pi = 314267958$, as the number of crosses in the shaded squares of the permutation array, for $i=3$. The darker square at the corner of the array on the left is counted twice. In this example, $h_3(\pi)=\ell_3(\pi)=2$.}
\label{fig:hi_li}
\end{figure}
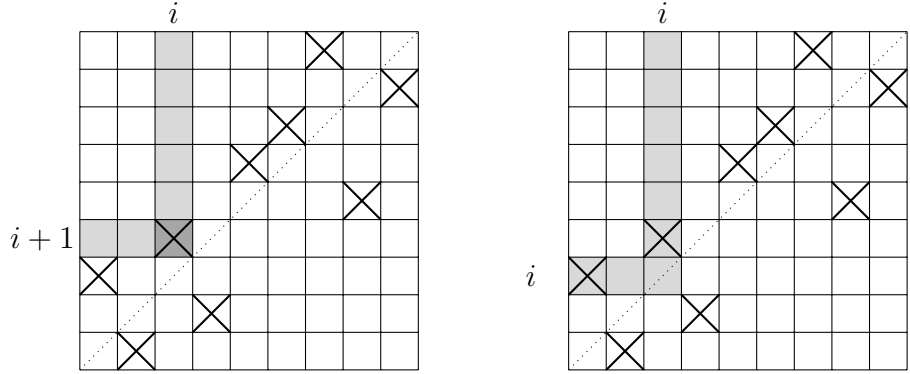

Hopkins and Joseph prove in \cite[Thm.\ 4.3]{HJ20}
that the statistic $h_i$ on antichains is $1$-mesic under $\rA$. This result can be translated in terms of $321$-avoiding permutations as follows.

\begin{theorem}[\cite{HJ20}]\label{thm:hi} 
For $1\le i\le n-1$, the statistic $h_i$ is $1$-mesic under the action of $\rS$ on $\S_n(321)$.
\end{theorem}

See Figure~\ref{fig:orbit} for a computation of $h_2(\pi)$ over an entire orbit for $n=4$ and one for $n=5$. As Theorem~\ref{thm:hi} asserts, the average of this statistic is 1 over each orbit.

\def\backgr{
\fill[brown!35!white] (0,1) rectangle (2,2);
\fill[brown!35!white] (1,2) rectangle (2,4);
\fill[pattern=north east lines,pattern color=violet] (0,0) rectangle (1,1);
\fill[pattern=north east lines,pattern color=violet] (1,1) rectangle (2,2);
\fill[pattern=north east lines,pattern color=violet] (2,2) rectangle (3,3);
\fill[pattern=north east lines,pattern color=violet] (3,3) rectangle (4,4);
 \draw (0,0) grid (4,4);
 \draw[dotted] (0,0)--(4,4);
}

\def\backgrr{
\fill[brown!35!white] (0,1) rectangle (2,2);
\fill[brown!35!white] (1,2) rectangle (2,5);
\fill[pattern=north east lines,pattern color=violet] (0,0) rectangle (1,1);
\fill[pattern=north east lines,pattern color=violet] (1,1) rectangle (2,2);
\fill[pattern=north east lines,pattern color=violet] (2,2) rectangle (3,3);
\fill[pattern=north east lines,pattern color=violet] (3,3) rectangle (4,4);
\fill[pattern=north east lines,pattern color=violet] (4,4) rectangle (5,5);
 \draw (0,0) grid (5,5);
 \draw[dotted] (0,0)--(5,5);
}

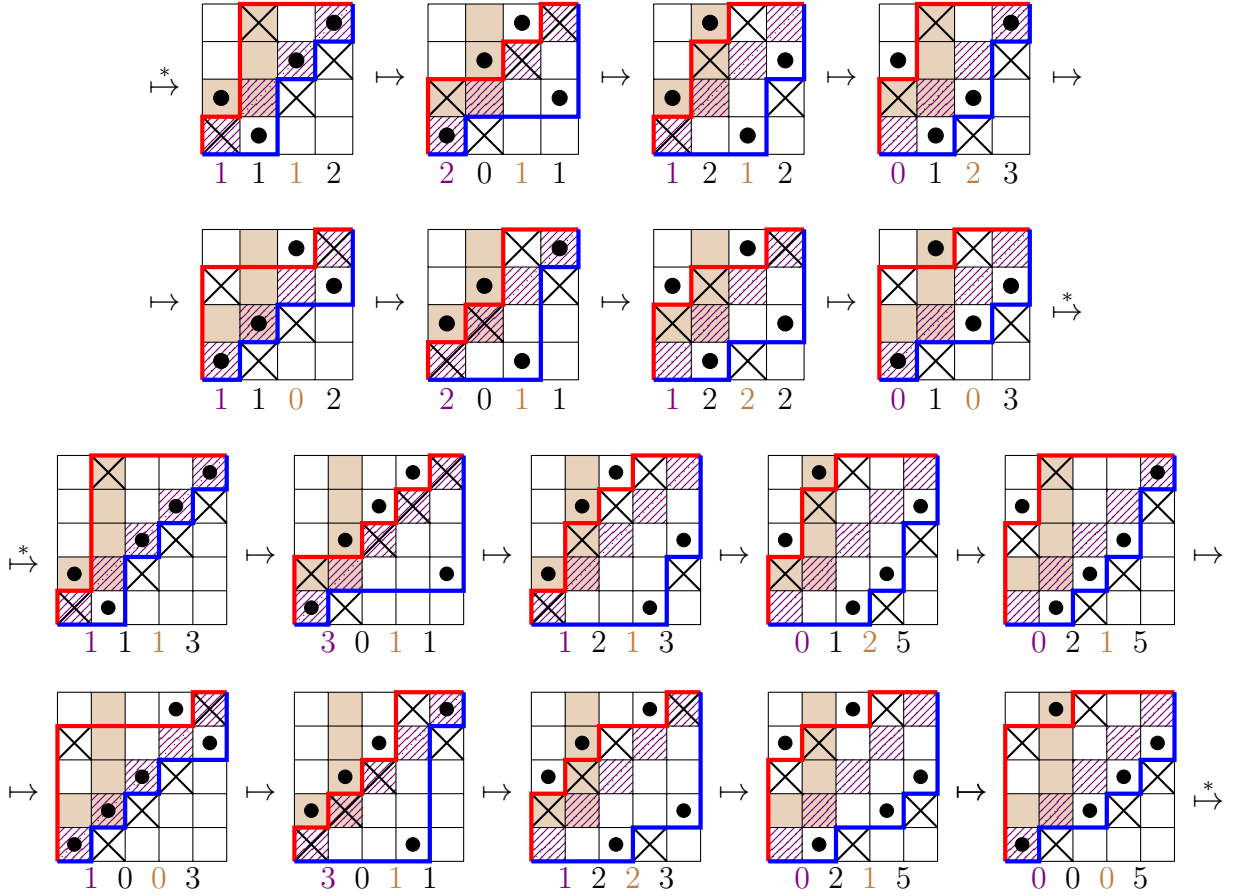
\begin{figure}[htb]
    \centering
\begin{tikzpicture}[scale=0.5]
\backgr
 \perm{{1,4,2,3}}
 \draw[red,ultra thick] (0,0) \N\E\N\N\N\E\E\E;
 \draw[blue,ultra thick] (0,0) \E\E\N\N\E\N\E\N;
 \permc{{2,1,3,4}}
 \draw (-1,2) node {$\stackrel{*}{\mapsto}$};
 \draw (5,2) node {$\mapsto$};
   \draw[brown] (2.5,-.5) node {1};
 \draw (1.5,-.5) node {1};
 \draw[violet] (.5,-.5) node {1};
 \draw (3.5,-.5) node {2};
\begin{scope}[shift={(6,0)}]
\backgr
 \perm{{2,1,3,4}}
 \draw[red,ultra thick] (0,0) \N\N\E\E\N\E\N\E;
 \draw[blue,ultra thick] (0,0) \E\N\E\E\E\N\N\N;
 \permc{{1,3,4,2}}
 \draw (5,2) node {$\mapsto$};
  \draw[brown] (2.5,-.5) node {1};
 \draw (1.5,-.5) node {0};
 \draw[violet] (.5,-.5) node {2};
 \draw (3.5,-.5) node {1};
 \end{scope}
 \begin{scope}[shift={(12,0)}]
\backgr
 \perm{{1,3,4,2}}
 \draw[red,ultra thick] (0,0) \N\E\N\N\E\N\E\E;
 \draw[blue,ultra thick] (0,0) \E\E\E\N\N\E\N\N;
 \permc{{2,4,1,3}}
 \draw (5,2) node {$\mapsto$};
 \draw[brown] (2.5,-.5) node {1};
 \draw (1.5,-.5) node {2};
 \draw[violet] (.5,-.5) node {1};
 \draw (3.5,-.5) node {2};
 \end{scope}
  \begin{scope}[shift={(18,0)}]
\backgr
 \perm{{2,4,1,3}}
 \draw[red,ultra thick] (0,0) \N\N\E\N\N\E\E\E;
 \draw[blue,ultra thick] (0,0) \E\E\N\E\N\N\E\N;
 \permc{{3,1,2,4}}
 \draw (5,2) node {$\mapsto$};
 \draw[brown] (2.5,-.5) node {2};
 \draw (1.5,-.5) node {1};
 \draw[violet] (.5,-.5) node {0};
 \draw (3.5,-.5) node {3};
 \end{scope}
\begin{scope}[shift = {(0,-6)}]
\backgr
 \perm{{3,1,2,4}}
 \draw[red,ultra thick] (0,0) \N\N\N\E\E\E\N\E;
 \draw[blue,ultra thick] (0,0) \E\N\E\N\E\E\N\N;
 \permc{{1,2,4,3}}
  \draw (-1,2) node {$\mapsto$};
 \draw (5,2) node {$\mapsto$};
  \draw[brown] (2.5,-.5) node {0};
 \draw (1.5,-.5) node {1};
 \draw[violet] (.5,-.5) node {1};
 \draw (3.5,-.5) node {2};
 \end{scope}
\begin{scope}[shift={(6,-6)}]
\backgr
 \perm{{1,2,4,3}}
 \draw[red,ultra thick] (0,0) \N\E\N\E\N\N\E\E;
 \draw[blue,ultra thick] (0,0) \E\E\E\N\N\N\E\N;
 \permc{{2,3,1,4}}
 \draw (5,2) node {$\mapsto$};
  \draw[brown] (2.5,-.5) node {1};
 \draw (1.5,-.5) node {0};
 \draw[violet] (.5,-.5) node {2};
 \draw (3.5,-.5) node {1};
 \end{scope}
 \begin{scope}[shift={(12,-6)}]
\backgr
 \perm{{2,3,1,4}}
 \draw[red,ultra thick] (0,0) \N\N\E\N\E\E\N\E;
 \draw[blue,ultra thick] (0,0) \E\E\N\E\E\N\N\N;
 \permc{{3,1,4,2}}
 \draw (5,2) node {$\mapsto$};
  \draw[brown] (2.5,-.5) node {2};
 \draw (1.5,-.5) node {2};
 \draw[violet] (.5,-.5) node {1};
 \draw (3.5,-.5) node {2};
 \end{scope}
  \begin{scope}[shift={(18,-6)}]
\backgr
 \perm{{3,1,4,2}}
 \draw[red,ultra thick] (0,0) \N\N\N\E\E\N\E\E;
 \draw[blue,ultra thick] (0,0) \E\N\E\E\N\E\N\N;
 \permc{{1,4,2,3}}
\draw (5,2) node {$\stackrel{*}{\mapsto}$};
 \draw[brown] (2.5,-.5) node {0};
 \draw (1.5,-.5) node {1};
 \draw[violet] (.5,-.5) node {0};
 \draw (3.5,-.5) node {3};
 \end{scope}
\end{tikzpicture}\bigskip

\begin{tikzpicture}[scale = .45]
\backgrr
 \perm{{1,5,2,3,4}}
 \draw[red,ultra thick] (0,0) \N\E\N\N\N\N\E\E\E\E;
 \draw[blue,ultra thick] (0,0) \E\E\N\N\E\N\E\N\E\N;
 \permc{{2,1,3,4,5}}
  \draw (-1,2) node {$\stackrel{*}{\mapsto}$};
 \draw (6,2) node {$\mapsto$};
    \draw[brown] (3,-.5) node {1};
 \draw (2,-.5) node {1};
 \draw[violet] (1,-.5) node {1};
 \draw (4,-.5) node {3};
 \begin{scope}[shift={(7,0)}]
\backgrr
 \perm{{2,1,3,4,5}}
 \draw[red,ultra thick] (0,0) \N\N\E\E\N\E\N\E\N\E;
 \draw[blue,ultra thick] (0,0) \E\N\E\E\E\E\N\N\N\N;
 \permc{{1,3,4,5,2}}
 \draw (6,2) node {$\mapsto$};
    \draw[brown] (3,-.5) node {1};
 \draw (2,-.5) node {0};
 \draw[violet] (1,-.5) node {3};
 \draw (4,-.5) node {1};
 \end{scope}
  \begin{scope}[shift={(14,0)}]
\backgrr
 \perm{{1,3,4,5,2}}
 \draw[red,ultra thick] (0,0) \N\E\N\N\E\N\E\N\E\E;
 \draw[blue,ultra thick] (0,0) \E\E\E\E\N\N\E\N\N\N;
 \permc{{2,4,5,1,3}}
 \draw (6,2) node {$\mapsto$};
    \draw[brown] (3,-.5) node {1};
 \draw (2,-.5) node {2};
 \draw[violet] (1,-.5) node {1};
 \draw (4,-.5) node {3};
 \end{scope}
   \begin{scope}[shift={(21,0)}]
\backgrr
 \perm{{2,4,5,1,3}}
 \draw[red,ultra thick] (0,0) \N\N\E\N\N\E\N\E\E\E;
 \draw[blue,ultra thick] (0,0) \E\E\E\N\E\N\N\E\N\N;
 \permc{{3,5,1,2,4}}
 \draw (6,2) node {$\mapsto$};
    \draw[brown] (3,-.5) node {2};
 \draw (2,-.5) node {1};
 \draw[violet] (1,-.5) node {0};
 \draw (4,-.5) node {5};
 \end{scope}
    \begin{scope}[shift={(28,0)}]
\backgrr
 \perm{{3,5,1,2,4}}
 \draw[red,ultra thick] (0,0) \N\N\N\E\N\N\E\E\E\E;
 \draw[blue,ultra thick] (0,0) \E\E\N\E\N\E\N\N\E\N;
 \permc{{4,1,2,3,5}}
 \draw (6,2) node {$\mapsto$};
    \draw[brown] (3,-.5) node {1};
 \draw (2,-.5) node {2};
 \draw[violet] (1,-.5) node {0};
 \draw (4,-.5) node {5};
 \end{scope}
 \begin{scope}[shift = {(0,-7)}]
\backgrr
 \perm{{4,1,2,3,5}}
 \draw[red,ultra thick] (0,0) \N\N\N\N\E\E\E\E\N\E;
 \draw[blue,ultra thick] (0,0) \E\N\E\N\E\N\E\E\N\N;
 \permc{{1,2,3,5,4}}
 \draw (-1,2) node {$\mapsto$};
 \draw (6,2) node {$\mapsto$};
    \draw[brown] (3,-.5) node {0};
 \draw (2,-.5) node {0};
 \draw[violet] (1,-.5) node {1};
 \draw (4,-.5) node {3};
  \end{scope}
 \begin{scope}[shift={(7,-7)}]
\backgrr
 \perm{{1,2,3,5,4}}
 \draw[red,ultra thick] (0,0) \N\E\N\E\N\E\N\N\E\E;
 \draw[blue,ultra thick] (0,0) \E\E\E\E\N\N\N\N\E\N;
 \permc{{2,3,4,1,5}}
 \draw (6,2) node {$\mapsto$};
    \draw[brown] (3,-.5) node {1};
 \draw (2,-.5) node {0};
 \draw[violet] (1,-.5) node {3};
 \draw (4,-.5) node {1};
 \end{scope}
  \begin{scope}[shift={(14,-7)}]
\backgrr
 \perm{{2,3,4,1,5}}
 \draw[red,ultra thick] (0,0) \N\N\E\N\E\N\E\E\N\E;
 \draw[blue,ultra thick] (0,0) \E\E\E\N\E\E\N\N\N\N;
 \permc{{3,4,1,5,2}}
 \draw (6,2) node {$\mapsto$};
    \draw[brown] (3,-.5) node {2};
 \draw (2,-.5) node {2};
 \draw[violet] (1,-.5) node {1};
 \draw (4,-.5) node {3};
 \end{scope}
   \begin{scope}[shift={(21,-7)}]
\backgrr
 \perm{{3,4,1,5,2}}
 \draw[red,ultra thick] (0,0) \N\N\N\E\N\E\E\N\E\E;
 \draw[blue,ultra thick] (0,0) \E\E\N\E\E\N\E\N\N\N;
 \permc{{4,1,5,2,3}}
 \draw (6,2) node {$\mapsto$};
    \draw[brown] (3,-.5) node {1};
 \draw (2,-.5) node {2};
 \draw[violet] (1,-.5) node {0};
 \draw (4,-.5) node {5};
 \end{scope}
    \begin{scope}[shift={(28,-7)}]
\backgrr
 \perm{{4,1,5,2,3}}
 \draw[red,ultra thick] (0,0) \N\N\N\N\E\E\N\E\E\E;
 \draw[blue,ultra thick] (0,0) \E\N\E\E\N\E\N\E\N\N;
 \permc{{1,5,2,3,4}}
  \draw (-1,2) node {$\mapsto$};
 \draw (6,2) node {$\stackrel{*}{\mapsto}$};
    \draw[brown] (3,-.5) node {0};
 \draw (2,-.5) node {0};
 \draw[violet] (1,-.5) node {0};
 \draw (4,-.5) node {5};
 \end{scope}
\end{tikzpicture}
    \caption{The rowmotion orbits containing 1423 (above) and 15234 (below). The numbers below each diagram are the values  $\textcolor{violet}{\fp(\pi)}$, $h_2(\pi)$, $\textcolor{brown}{\ell_2(\pi)}$ and  $\inv(\pi)$, from left to right.}
    \label{fig:orbit}
\end{figure}

Next we define a new family of permutation statistics, that we denote by $\ell_i$ for $1\le i \le n$. For $\pi\in\S_n(321)$, let $\ell_i(\pi)$ be the number of crosses in the array of $\pi$ of the form $(j,i)$ with $1\le j\le i$, plus the number of crosses of the form $(i,j)$ with $i< j\le n$. 
See the right of Figure~\ref{fig:hi_li} for a visualization of this statistic. We can also write
\begin{equation}\label{def:li}\ell_i(\pi)=\ind_{\pi^{-1}(i)\le i}+\ind_{\pi(i)>i}.\end{equation}

\begin{theorem}\label{thm:li}
		For $1\le i\le n$, the statistic $\ell_i$ is $1$-mesic under the action of $\rho_{\S}$ on $\S_n(321)$. 
	\end{theorem}
 To prove this theorem, we will provide an alternative formulation of $\ell_i$ and use the following lemma.

	\begin{lemma}\label{UpShift}
		Let $\pi\in \S_n(321)$ and $\sigma = \rS(\pi)$. For all $1\le i \le n-1$, 
		\[ \ind_{\pi^{-1}(i)\le i} = \ind_{\sigma^{-1}(i+1)<i+1}, \]
		and for all $2\le i \le n$, 
		\[ \ind_{\pi(i)\ge i} = \ind_{\sigma(i-1)>i-1}. \]
	\end{lemma}
	\begin{proof}
		 The left-hand side of the first equality equals one if $\pi$ has a weak excedance in row $i$ of the array, which is equivalent to the path $E_p(\pi)$ having a peak in row $i$. This happens if and only if this path, which equals $E_v(\sigma)$ by Equation~\eqref{eq:rS_paths}, has a valley in row $i+1$. But this is equivalent to $\sigma$ having an excedance in row $i+1$, and also to the right-hand side being equal to one.
		
		The second equality is proved similarly using columns instead of rows.
	\end{proof}

\begin{proof}[Proof of Theorem \ref{thm:li}]
    The statement is trivial when $i\in\{1,n\}$, as in this case $\ell_i(\pi)=1$ for all $\pi\in\S_n$. 
    Suppose now that $2 \le i \le n-1$, and let $\mathcal{O}$ be an arbitrary orbit of $\rS$ on $\S_n(321)$. 
    By Equation~\eqref{def:li} and the first part of Lemma~\ref{UpShift},
    \begin{align*}\sum_{\pi\in \mathcal{O}}\ell_i(\pi)&=\sum_{\pi\in \mathcal{O}}\ \ind_{\pi^{-1}(i)\le i}+\sum_{\pi\in \mathcal{O}} \ind_{\pi(i)>i} =\sum_{\sigma\in \mathcal{O}}\ind_{\sigma^{-1}(i+1) <i+1}+\sum_{\pi\in \mathcal{O}} \ind_{\pi(i)>i}\\
    &=\sum_{\pi\in \mathcal{O}}\left(\ind_{\pi^{-1}(i+1) <i+1}+\ind_{\pi(i)>i}\right)=\sum_{\pi \in \mathcal{O}} h_i(\pi),\end{align*} using Equation~\eqref{def:hi} in the last step. Since $h_i$ is 1-mesic by Theorem \ref{thm:hi}, it follows that so is~$\ell_i$.
\end{proof}

	Theorem \ref{thm:li} yields an additional proof that the fixed point statistic on $\S_n(321)$ is 1-mesic under $\rho_{\S}$. 
 	
 	\begin{proof}[Second proof of Theorem \ref{thm:fp}]
We claim that, for any $\pi\in\S_n(321)$,
\[\fp(\pi)=\sum_{i=1}^n \ell_i(\pi)-\sum_{i=1}^{n-1} h_i(\pi).\]
  Indeed, the sum $\sum_{i=1}^n \ell_i(\pi)$ counts every excedance of $\pi$ twice and each fixed point once, whereas $\sum_{i=1}^{n-1} h_i(\pi)$ counts every excedance twice. Thus, by Theorems~\ref{thm:hi} and~\ref{thm:li}, the average of $\fp$ over each rowmotion orbit is $n-(n-1)=1$. 
 	\end{proof}

	Even though one cannot directly express $\ell_i(\pi)$ or $\fp(\pi)$ as a combination of indicator functions on the antichain $\Exc(\pi)$, we can still translate these statistics into statistics on antichains. 
	
	Let us start with the statistic $\fp$. Let $\sigma=\rS(\pi)$, and let $A=\Exc(\sigma)$. If $2\le i\le n-1$, then $(i,i)$ is a fixed point of $\pi$ if and only if  both $(i-1,i)$ and $(i,i+1)$ are excedances of $\sigma$. Additionally, by Lemma~\ref{UpShift}, $(1,1)$ is a fixed point of $\pi$ if and only if $\sigma(1)=2$, and $(n,n)$ is a fixed point of $\pi$ if and only if $\sigma(n-1)=n$. It follows that
	$$\fp(\pi)=\ind_{\sigma(1)=2}+\ind_{\sigma(1)=2\wedge\sigma(2)=3}+\dots+\ind_{\sigma(n-2)=n-1\wedge\sigma(n-1)=n}+\ind_{\sigma(n-1)=n}.$$
	Using now that an excedance $(a,b)$ of $\sigma$ corresponds to the element $[a,b-1]$ in $A=\Exc(\sigma)$ by Definition~\ref{def:Exc}, the statistic $\fp$ on the permutation $\pi$ translates into the following statistic on the antichain $A$:
	\begin{equation}\label{eq:fpA}\one_{[1,1]}+\min(\one_{[1,1]},\one_{[2,2]})+\dots+\min(\one_{[n-2,n-2]},\one_{[n-1,n-1]})+\one_{[n-1,n-1]}.	\end{equation}
	
	Let us now translate the statistic $\ell_i$ into a statistic on antichains. If $i=1$ or $i=n$, then $\ell_i$ is always equal to $1$, so let us assume that $2\le i\le n-1$. 
Again letting $\sigma=\rS(\pi)$, Equation~\eqref{def:li} and Lemma~\ref{UpShift} give
$$\ell_i(\pi)=\ind_{\pi^{-1}(i)\le i}+\ind_{\pi(i)\ge i}-\ind_{\pi(i)=i}
=\ind_{\sigma^{-1}(i+1)<i+1}+\ind_{\sigma(i-1)>i-1}-\ind_{\sigma(i-1)=i\wedge\sigma(i)=i+1}.
$$
Thus, the statistic $\ell_i$ on $\pi$ translates into the following statistic on $A=\Exc(\sigma)$:
\begin{align}\nonumber &\sum_{j=1}^{i}\one_{[j,i]}+\sum_{j=i-1}^{n-1}\one_{[i-1,j]}-\min(\one_{[i-1,i-1]},\one_{[i,i]})\\
\nonumber &\quad =\sum_{j=1}^{i-1}\one_{[j,i]}+\sum_{j=i}^{n-1}\one_{[i-1,j]}+\one_{[i,i]}+\one_{[i-1,i-1]}-\min(\one_{[i-1,i-1]},\one_{[i,i]})\\
\label{eq:liA} &\quad=\sum_{j=1}^{i-1}\one_{[j,i]}+\sum_{j=i}^{n-1}\one_{[i-1,j]}+\max(\one_{[i-1,i-1]},\one_{[i,i]}),
\end{align}
using that $x+y=\max(x,y)+\min(x,y)$.

\subsection{The sign statistic}
Next we describe how rowmotion on $321$-avoiding permutations interacts with the sign statistic. The sign of a permutation $\pi$ can be defined as $\sgn(\pi)=(-1)^{\inv(\pi)}$, where $\inv(\pi)=|\{(i,i'):i<i'\text{ and }\pi(i)>\pi(i')\}|$ 
is the number of inversions of $\pi$.

\begin{theorem}\label{SignRowmotionTheorem}
For all $\pi\in\S_n(321)$,
\[\sgn(\rS(\pi))=\sgn(\LKS(\pi))=\begin{cases}
\sgn(\pi) & \text{if $n$ is odd,}\\
-\sgn(\pi) & \text{if $n$ is even.}
 \end{cases}\]
\end{theorem}

\begin{proof}
Let us first prove the statement about $\rS$.
Let $\pi\in\S_n(321)$.
If $(i,i')$ is an inversion of $\pi$, then $(i,\pi(i))$ must be an excedance and $(i',\pi(i'))$ must be a deficiency. In addition, for any given excedance $(i,\pi(i))$, the number of inversions of the form $(i,i')$ is equal to $\pi(i)-i$. It follows that if the weak excedances of $\pi$ are $(i_1,j_1),(i_2,j_2),\dots,(i_r,j_r)$ with $i_1<i_2<\dots<i_r$, then
$$\inv(\pi)=\sum_{k=1}^r (j_k-i_k).$$

Let $\sigma=\rS(\pi)$. Using the description~\eqref{eq:rS_paths}, the excedances of $\sigma$ are in the positions of the valleys of $E_p(\pi)$, which are $(i_2-1,j_1+1),(i_3-1,j_2+1),\dots,(i_r-1,j_{r-1}+1)$. 
It follows that
$$\inv(\sigma)=\sum_{k=1}^{r-1} ((j_k+1)-(i_{k+1}-1))=\sum_{k=1}^r (j_k-i_k)-j_r+i_1+2(r-1)=\inv(\pi)-n+1+2(r-1),$$
which has the same parity as $\inv(\pi)$ if $n$ is odd, and the opposite parity if $n$ is even. This proves the statement for the map $\rS$.

It follows from Lemma \ref{LK_of_Perm} that $\LKS(\sigma)=(\rS(\sigma))^{-1}$ for all $\sigma\in\S_n(321)$. Taking the inverse of a permutation preserves the number of inversions, and hence the statistic $\sgn$. Therefore,
\[\sgn(\LKS(\pi))=\sgn((\rS(\pi))^{-1})=\sgn(\rS(\pi)).\qedhere\]
\end{proof}

A permutation is said to be odd or even according to the parity of its number of inversions.
It is a classical result of Simion and Schmidt \cite[Prop.\ 2]{SS85} 
that, when $n$ is even, the set $\S_n(321)$ contains the same number of odd and even permutations. A bijective proof of this fact was given by Reifegerste~\cite{RE05}. Theorem~\ref{SignRowmotionTheorem} gives two new bijections, $\rS$ and $\LKS$, between the subsets of odd and even permutations in $\S_n(321)$. Furthermore, $\LKS$ has the additional property of being a sign-reversing involution on $\S_n(321)$. 

Finally, we note the following two immediate consequences of Theorem~\ref{SignRowmotionTheorem}.
\begin{corollary}
For even $n$, the statistic $\sgn$ on $\S_n(321)$ is $0$-mesic under the action of $\rS$.
\end{corollary}

\begin{corollary}\label{NoLKFixed}
For even $n$, the map $\LK$ on $\D_n$ has no fixed points.
\end{corollary}
Corollary \ref{NoLKFixed} is equivalent to \cite[Thm.\ 4.6]{PAN04} for even~$n$.

\section{The Armstrong--Stump--Thomas bijection as a map on permutations}\label{sec:AST}

In~\cite{AST13}, Armstrong, Stump and Thomas constructed a bijection between antichains in root posets of finite Weyl groups and noncrossing matchings, having the property that it translates rowmotion on antichains into rotation of noncrossing matchings. In type $A$, we can interpret antichains as $321$-avoiding permutations via the bijection $\Exc$. In this section we will show the surprising fact that, with this interpretation, the Armstrong--Stump--Thomas bijection is equivalent to the well-known Robinson--Schensted--Knuth correspondence restricted to $321$-avoiding permutations. We will also show that this fact extends to the root poset in type~$B$.

\subsection{$\AST$ in type $A$}

The Armstrong--Stump--Thomas bijection is described in \cite{AST13} in much more generality, but for our purposes  we will be restricting exclusively to types $A$ and $B$. We follow the description given by Defant and Hopkins~\cite{DH21}. In type $A$, we denote this bijection by $\AST$. Recall the definition of $\NC_n$ from Section~\ref{sec:Dyck}. Throughout the section, for matchings in $\NC_n$, it will be convenient to denote the vertex $2n+1-i$ by $\ol{i}$, for each $1\le i\le n$, so that the vertices are $1,2,\dots, n, \ol{n}, \ol{n-1},\dots, \ol{1}$ in clockwise direction. 

\begin{definition}[\cite{AST13,DH21}]\label{def:AST}
Let $\AST: \cA(\A^{n-1})\to \NC_{n}$ be the bijection where the image of $A\in\cA(\A^{n-1})$ is the matching obtained as follows.
For each $i$ from $1$ to $n$, consider two options:
\begin{itemize}
    \item if $[i,j-1]\in A$ for some $j$, match the vertex $j$ with the nearest unmatched vertex in  counterclockwise direction;
    \item otherwise, match the vertex $\ol{i}$ with the nearest unmatched vertex in clockwise direction.
\end{itemize}
\end{definition}

See Figure~\ref{fig:ASTExample} for an example of this construction. Now we can state the main result of this section.

\begin{figure}[h]
    \centering
    \begin{tikzpicture}
        \typeA{3}
\an{1}{3}\an{3}{4}
\draw (1.1,1.2) node[right,scale=.7] {$[1,3]$};
\draw (2.6,.6) node[right,scale=.7] {$[3,4]$};
\draw[-to] (3,1) -- (3.4,1);
\begin{scope}[shift={(5.5,1)}]
\foreach \i in {1,...,5} {
\filldraw (90+18-\i*36:1.4) circle (.1);
 \draw (90+18-\i*36:1.7) node[scale=.7] {$\i$};
 \filldraw (90-18+\i*36:1.4) circle (.1);
 \draw (90-18+\i*36:1.7) node[scale=.7] {$\bar{\i}$};
}
\draw[-to] (2,0) -- (2.4,0);
\match{3}{4}
\end{scope}

\begin{scope}[shift={(10,1)}]
\foreach \i in {1,...,5} {
\filldraw (90+18-\i*36:1.4) circle (.1);
 \draw (90+18-\i*36:1.7) node[scale=.7] {$\i$};
 \filldraw (90-18+\i*36:1.4) circle (.1);
 \draw (90-18+\i*36:1.7) node[scale=.7] {$\bar{\i}$};
}
\draw[-to] (2,0) -- (2.4,0);
\match{3}{4}
\match{9}{10} 
\end{scope}

\begin{scope}[shift={(3,-3)}]
\foreach \i in {1,...,5} {
\filldraw (90+18-\i*36:1.4) circle (.1);
 \draw (90+18-\i*36:1.7) node[scale=.7] {$\i$};
 \filldraw (90-18+\i*36:1.4) circle (.1);
 \draw (90-18+\i*36:1.7) node[scale=.7] {$\bar{\i}$};
}
\draw[-to] (2,0) -- (2.4,0);
\match{3}{4}
\match{9}{10} 
\match{2}{5}
\end{scope}

\begin{scope}[shift={(7.5,-3)}]
\foreach \i in {1,...,5} {
\filldraw (90+18-\i*36:1.4) circle (.1);
 \draw (90+18-\i*36:1.7) node[scale=.7] {$\i$};
 \filldraw (90-18+\i*36:1.4) circle (.1);
 \draw (90-18+\i*36:1.7) node[scale=.7] {$\bar{\i}$};
}
\draw[-to] (2,0) -- (2.4,0);
\match{3}{4}
\match{9}{10} 
\match{2}{5}
\match{7}{8}
\end{scope}

\begin{scope}[shift={(12,-3)}]
\foreach \i in {1,...,5} {
\filldraw (90+18-\i*36:1.4) circle (.1);
 \draw (90+18-\i*36:1.7) node[scale=.7] {$\i$};
 \filldraw (90-18+\i*36:1.4) circle (.1);
 \draw (90-18+\i*36:1.7) node[scale=.7] {$\bar{\i}$};
}
\match{3}{4}
\match{9}{10} 
\match{2}{5}
\match{1}{6}
\match{7}{8}
\end{scope}
    \end{tikzpicture}
    \caption{The computation of $\AST(\{[1,3],[3,4]\})$ in $\A^4$ using Definition~\ref{def:AST}.}
    \label{fig:ASTExample}
\end{figure}
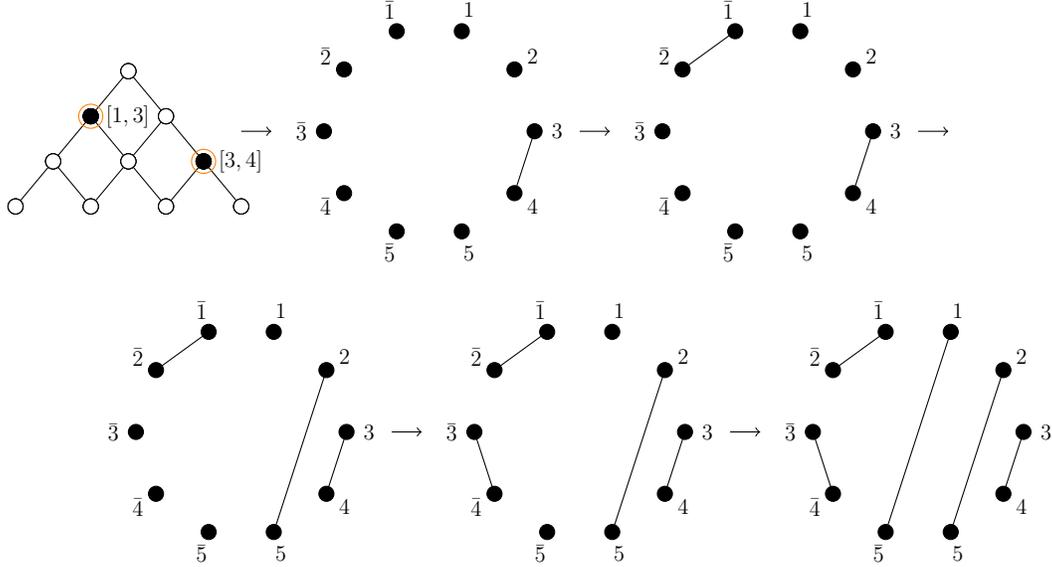

\begin{theorem}\label{ASTConnection} 
Let $\Match$, $\RSKD$, $\Exc$ and $\AST$ be the bijections from Definitions \ref{def:Match}, \ref{def:RSKD}, \ref{def:Exc} and \ref{def:AST}, respectively. Then $$\AST = \Match\circ \RSKD \circ \Exc^{-1}.$$
\end{theorem}

As shown in~\cite{AST13}, the bijection $\AST$ is equivariant in the sense that 
\begin{equation}\label{eq:equivariantAST}
    \AST\circ\rA=\Rot\circ\AST.
\end{equation}
Using the description in Theorem~\ref{ASTConnection} of $\AST$ as a composition, the commutative diagram in Figure~\ref{fig:ASTrow} shows how the action of rowmotion can be interpreted at each step of the composition. For examples of these maps, see Figure~\ref{fig:CommutativeDiagramExample}.

\begin{figure}[h]
    \centering
    \begin{tikzpicture}
        \draw node (A0) at (-6,0) {$\cA(\A^{n-1})$};
        \draw node (A1) at (-6,-2) {$\cA(\A^{n-1})$};
        \draw node (S0) at (-3,0) {$\S_n(321)$};
        \draw node (S1) at (-3,-2) {$\S_n(321)$};
        \draw node (D0) at (0,0) {$\D_n$};
        \draw node (D1) at (0,-2) {$\D_n$};
        \draw node (M0) at (3,0) {$\NC_n$};
        \draw node (M1) at (3,-2) {$\NC_n$};
        \draw[->] (D0) -- node[right]{$\Pro^{-1}$} (D1);
        \draw[->] (M0) -- node[right]{$\Rot$} (M1);
        \draw[->] (S0) -- node[right]{$\rS$} (S1);
        \draw[->] (A0) -- node[right]{$\rA$} (A1);
        \draw[->] (S0) -- node[above]{$\Exc$} (A0);
        \draw[->] (S0) -- node[above]{$\RSKD$} (D0);        
        \draw[->] (D0) -- node[above]{$\Match$} (M0);
        \draw[->] (S1) -- node[below]{$\Exc$} (A1);
        \draw[->] (S1) -- node[below]{$\RSKD$} (D1);        
        \draw[->] (D1) -- node[below]{$\Match$} (M1);
    \draw[->] (-5.8,.3) to [bend left =15] node[above]{$\AST$} (2.8,.3);
    \draw[->] (-5.8,-2.3) to [bend left =-15] node[below]{$\AST$} (2.8,-2.3);
    \end{tikzpicture}
\caption{The interaction of the maps in Theorem~\ref{ASTConnection} with rowmotion.}
    \label{fig:ASTrow}
\end{figure}
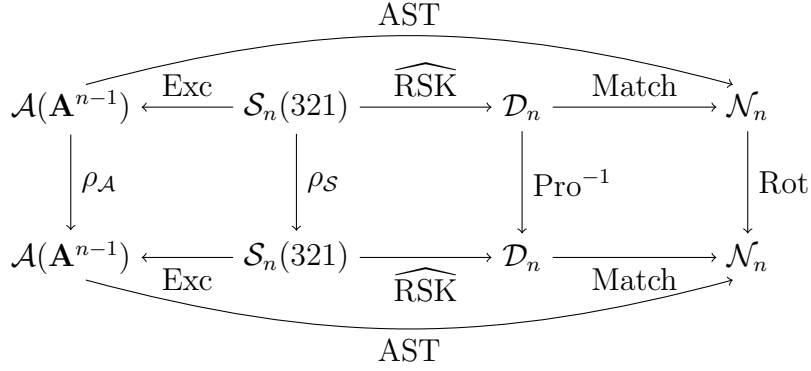

\begin{figure}[htb]
    \centering
    \begin{tikzpicture}[scale=1.05]
\draw[->] (3.5,1.3) to [bend left =15] node[above]{$\AST$} (14.3,1.7);
  \begin{scope}[shift={(2.25,-.5)},scale=.8]
\typeA{3}
\an{1}{2}\an{2}{4}
\draw[-to] (1.5,-1) --node[right]{$\rA$} (1.5,-2);
\end{scope}
  \begin{scope}[shift={(5.5,-1)},scale=.5]
	  \draw (0,0) grid (5,5);
 \draw[dotted] (0,0)--(5,5);
 \perm{{3,5,1,2,4}}
  \draw[red,ultra thick] (0,0) \N\N\N\E\N\N\E\E\E\E;
  \draw[-to] (-0.5, 2.66) --node[above]{$\Exc$} (-1.5,2.66);
\draw[-to] (2.5,-.5) --node[right]{$\rS$} (2.5,-2);
\end{scope}
\begin{scope}[shift={(9,0)},scale=.7]
\coordinate (a) at (1,0);
\coordinate (b) at (.5,.5);
\filldraw[thick] (0,0)\start\up\up\dn\up\dn\dn\up\up\dn\dn;
\draw[-to ] (-1.25, .5) --node[above]{$\RSKD$} (-.25,.5);
\draw[-to ] (5.5, .5) --node[above]{$\Match$} (6.75,.5);
\draw[-to] (2.5,-1.5) --node[right]{$\Pro^{-1}$} (2.5,-3);
\end{scope}
\begin{scope}[shift={(15.5,.4)},scale=.8]
\foreach \i in {1,...,5} {
\filldraw (90+18-\i*36:1.4) circle (.1);
 \draw (90+18-\i*36:1.7) node[scale=.7] {$\i$};
 \filldraw (90-18+\i*36:1.4) circle (.1);
 \draw (90-18+\i*36:1.7) node[scale=.7] {$\bar{\i}$};
}
 \match{1}{6}\match{2}{3}\match{4}{5}\match{7}{10}\match{8}{9}
 \draw[-to] (0,-2) --node[right]{$\Rot$} (0,-3);
\end{scope}
  \begin{scope}[shift={(2.25,-4.5)},scale=.8]
\typeA{3}
\an{1}{3}
\draw[-to] (1.5,-1) --node[right]{$\rA$} (1.5,-2);
\end{scope}

  \begin{scope}[shift={(5.5,-5)},scale=.5]
	  \draw (0,0) grid (5,5);
 \draw[dotted] (0,0)--(5,5);
 \perm{{4,1,2,3,5}}
  \draw[red,ultra thick] (0,0) \N\N\N\N\E\E\E\E\N\E;
  \draw[-to] (-0.5, 2.66) --node[above]{$\Exc$} (-1.5,2.66);
\draw[-to] (2.5,-.5) --node[right]{$\rS$} (2.5,-2);
\end{scope}
\begin{scope}[shift={(9,-4)},scale=.7]
\coordinate (a) at (1,0);
\coordinate (b) at (.5,.5);
\filldraw[thick] (0,0)\start\up\up\up\dn\up\dn\dn\dn\up\dn;
\draw[-to ] (-1.25, .5) --node[above]{$\RSKD$} (-.25,.5);
\draw[-to ] (5.5, .5) --node[above]{$\Match$} (6.75,.5);
\draw[-to] (2.5,-1.5) --node[right]{$\Pro^{-1}$} (2.5,-3);
\end{scope}
\begin{scope}[shift={(15.5,-3.6)},scale=.8]
\foreach \i in {1,...,5} {
\filldraw (90+18-\i*36:1.4) circle (.1);
 \draw (90+18-\i*36:1.7) node[scale=.7] {$\i$};
 \filldraw (90-18+\i*36:1.4) circle (.1);
 \draw (90-18+\i*36:1.7) node[scale=.7] {$\bar{\i}$};
}
\match{1}{8}\match{9}{10}\match{4}{3}\match{7}{2}\match{5}{6}
\draw[-to] (0,-2) --node[right]{$\Rot$} (0,-3);
\end{scope}
  \begin{scope}[shift={(2.25,-8.5)},scale=.8]
\typeA{3}
\an{4}{4}
  \draw (1.5,-1) node {$\cA(\A^4)$};
\end{scope}
  \begin{scope}[shift={(5.5,-9)},scale=.5]
	  \draw (0,0) grid (5,5);
 \draw[dotted] (0,0)--(5,5);
 \perm{{1,2,3,5,4}}
  \draw[red,ultra thick] (0,0) \N\E\N\E\N\E\N\N\E\E;
  \draw[-to] (-0.5, 2.66) --node[above]{$\Exc$} (-1.5,2.66);
    \draw (2.5,-.8) node {$\S_5(321)$};
\end{scope}
\begin{scope}[shift={(9,-8)},scale=.7]
\coordinate (a) at (1,0);
\coordinate (b) at (.5,.5);
\filldraw[thick] (0,0)\start\up\up\up\up\dn\up\dn\dn\dn\dn;
\draw[-to ] (-1.25, .5) --node[above]{$\RSKD$} (-.25,.5);
\draw[-to ] (5.5, .5) --node[above]{$\Match$} (6.75,.5);
    \draw (2.5,-2) node {$\D_5$};
\end{scope}
\begin{scope}[shift={(15.5,-7.6)},scale=.8]
\foreach \i in {1,...,5} {
\filldraw (90+18-\i*36:1.4) circle (.1);
 \draw (90+18-\i*36:1.7) node[scale=.7] {$\i$};
 \filldraw (90-18+\i*36:1.4) circle (.1);
 \draw (90-18+\i*36:1.7) node[scale=.7] {$\bar{\i}$};
}
\match{1}{10}\match{9}{2}\match{8}{3}\match{7}{6}\match{4}{5}
    \draw (0,-2.2) node {$\NC_5$};
\end{scope}
\end{tikzpicture}
    \caption{An example of Theorem~\ref{ASTConnection} applied to three elements of a rowmotion orbit. }
    \label{fig:CommutativeDiagramExample}
\end{figure}
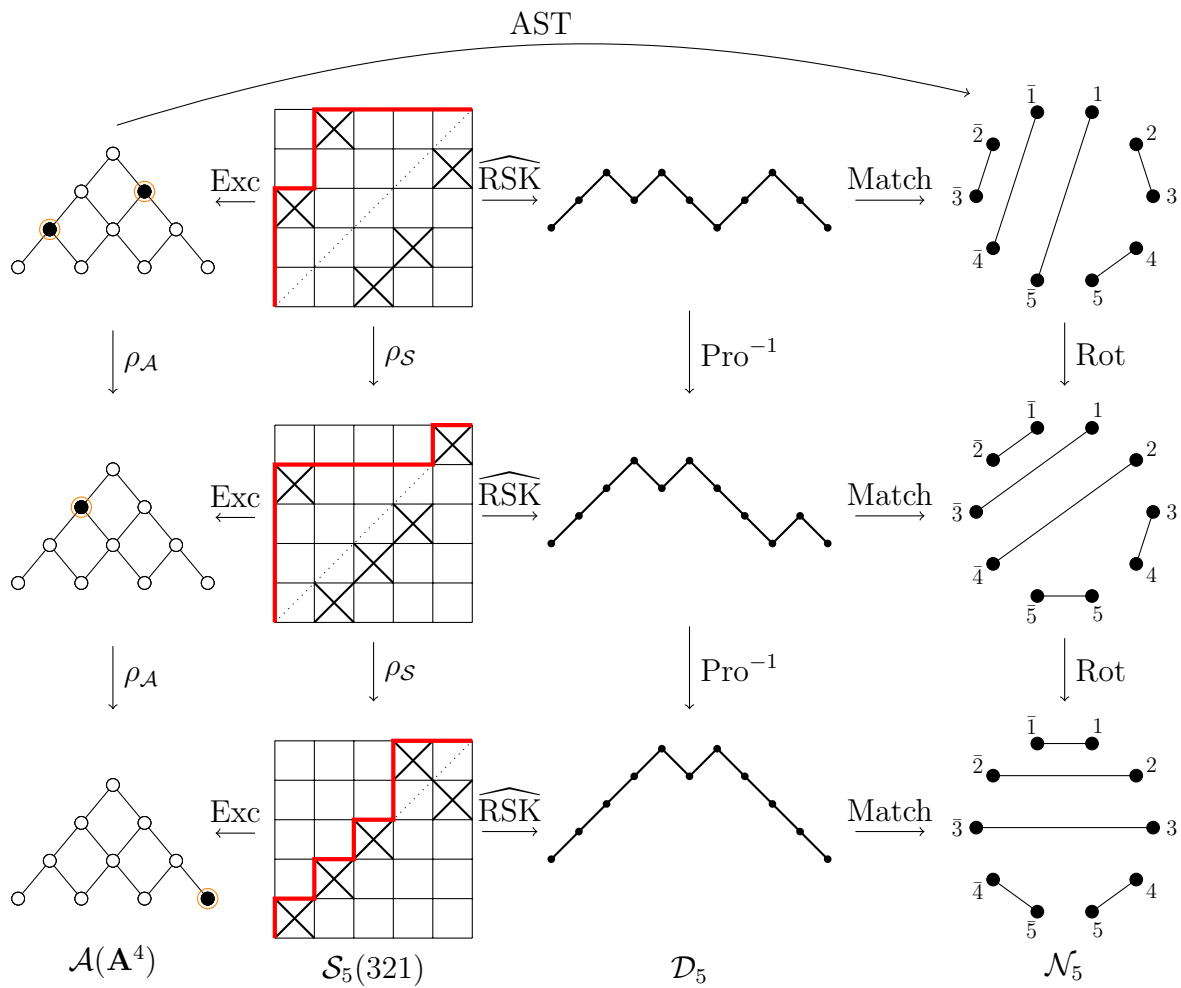

Before proving Theorem~\ref{ASTConnection}, we establish a property of noncrossing matchings that will be useful in the proof.
Given an arc of $M\in\NC_n$ with endpoints $x$ and $y$, where $1\le x<y\le 2n$, we define the {\em bottom} of the arc to be $y$ if $x+y\le 2n+1$, and $x$ otherwise. Note that when placing the points as described in Section~\ref{sec:Dyck}, the bottom is the lower endpoint, breaking ties by taking the left endpoint in that case. For $M\in\NC_n$, let $B(M)$ be the set of $n$ vertices that are bottoms of some arc in $M$. 
A vertex of the form $\ol{i}$ is identified with $2n+1-i$ in order to use the above definition,

\begin{lemma}\label{lem:UniqueMatchingData}
    Let $M_1,M_2\in \NC_n$ and suppose that $B(M_1)=B(M_2)$. Then $M_1=M_2$.
\end{lemma}
\begin{proof}
    Let $B=B(M_1)=B(M_2)$, and denote the bottoms in the right and left halves by $R=B\cap\{1,2,\dots,n\}$ and $L=B\cap\{\ol{n},\ol{n-1},\dots,\ol{1}\}$. Next we construct a matching $M$ that satisfies $B(M)=B$. First,
    label the elements of $B$ by $b_1,b_2,\dots,b_n$ as follows: label each $\ol{i}\in L$ with $b_i$, then assign the remaining labels to the elements of $R$ in increasing order of the indices.
    Now, for $i$ from $1$ to $n$, match $b_i$ to the closest (along the boundary of the circle) unmatched vertex which lies weakly above it. Such a vertex exists because, if $b_i=\ol{i}\in L$, then at most $2(i-1)$ of the $2i-1$  other vertices which lie weakly above $\ol{i}$ have been matched, so at least one is unmatched. Similarly, if $b_i\in R$, then $b_i>i$ by construction, and so at least one of the $2(b_i-1)>2(i-1)$ vertices that lie above $b_i$ is unmatched. Additionally, matching each bottom to the nearest unmatched vertex guarantees that $M$ is noncrossing. 
    
    Next we show that $M$ is the unique matching in $\NC_n$ with this set of bottoms. Suppose for contradiction that $M_1\in\NC_n$ satisfies $B(M_1)=B$ and $M_1\neq M$, and let $b_i$ be the first bottom in the ordering which is matched to a different vertex in $M$ and $M_1$. Suppose that this bottom is matched to vertex $v$ in $M$ and to vertex $v'\neq v$ in $M_1$. Then there must be some bottom $b_j$ with $j>i$ which is matched to vertex $v$ in $M_1$. But then the arc $(b_j,v)$ would cross the arc $(b_i,v')$, reaching a contradiction. It follows that $M_1=M=M_2$.
\end{proof}

We will prove Theorem~\ref{ASTConnection} in the equivalent form $\AST\circ\Exc=\Match\circ\RSKD$. 
The next two lemmas describe the behavior of these bijections from $\S_n(321)$ to $\NC_n$. 
If $(i,j)$ is an excedance of a permutation, we refer to $i$ as the {\em position} and to $j$ as the {\em value}.

\begin{lemma}\label{PermViaAST}
    Let $\pi\in \S_n(321)$. Then 
    \begin{multline*}B(\AST(\Exc(\pi)))=\{j: j \text{ is the value of an excedance of }\pi\}\\
    \cup\{\ol{i}: 1\le i\le n \text{ and $i$ is not the position of an excedance of }\pi\}.\end{multline*}
\end{lemma}
\begin{proof}
It is enough to show that the right-hand side is contained in the left-hand side, since both sets have $n$ elements.
  Let $A = \Exc(\pi)$. By Definition~\ref{def:Exc}, the excedances of $\pi$ are $(i,j)$ for each $[i,j-1]\in A$. 
Let us show that, for each such $j$, we have $j\in B(\AST(A))$. Indeed, when $j$ is being matched in the computation of $\AST(A)$, only $i-1$ arcs have already been placed at this point. Thus, since $i<j$, not all $2i$ vertices of the form $k$ or $\ol{k}$ where $k\le i$, which are the ones nearest to $j$ in counterclockwise direction, have been already matched.  

Now suppose that $i$ is not the position of an excedance of $\pi$. We will show that  $\ol{i}$ is matched to some $k$ or $\ol{k}$ with $k\le i$. In this case, when $\ol{i}$ is being matched in the computation of $\AST(A)$, the $2(i-1)$ endpoints of the $i-1$ arcs that have already been placed cannot contain all the vertices in $\{\ol{i-1},\ol{i-2},\dots,\ol{1},1,2,\dots,i\}$, which are the $2i-1$ vertices nearest to $\ol{i}$ in clockwise direction, so $\ol{i}$ must be matched to one of these.
\end{proof}

\begin{lemma}\label{PermViaRSKD}
Let $\pi \in \S_n(321)$. Then 
  \begin{multline*}B(\Match(\RSKD(\pi)))=\{j: j \text{ is the value of an excedance of }\pi\}\\ 
    \cup\{2n+1-i: 1\le i\le n \text{ and $i$ is not the position of an excedance of }\pi\}.\end{multline*}
\end{lemma}
\begin{proof}
It is enough to show that the right-hand side is contained in the left-hand side, since both sets have $n$ elements.
     Recall that $(x,y)$ is a left tunnel of $D\in\D_n$ if $x,y\le n$, a right tunnel if $x,y\ge n+1$, and otherwise it is a centered tunnel if $x+y=2n+1$, a left-across tunnel if $x+y<2n+1$, and a right-across tunnel if $x+y>2n+1$. 
     By Definition~\ref{def:Match}, we obtain $\Match(D)$ by matching every pair $(x,y)$ such that the steps of $D$ in these positions form a tunnel.
     Assuming that $x<y$, it follows that if $(x,y)$ is a left, centered, or a left-across 
    tunnel of $D$, then $y$ is a bottom of $\Match(D)$. Similarly, if $(x,y)$ is a right or a right-across 
    tunnel, then $x$ is a bottom of $\Match(D)$. 

Now let $(P,Q)=\RSK(\pi)$, and let $D = \RSKD(\pi)$. 
First we show that 
the values of the excedances of $\pi$ which are bumped when applying $\RSK$ are precisely 
the positions of the right endpoints of the left tunnels of $D$.
Indeed, each down step in the first half of $D$ comes from an entry $j$ in the second row of $P$, which must have been bumped by a smaller entry to its right in $\pi$. Since $321$-avoiding permutations can be decomposed into two increasing sequences consisting of excedances and weak deficiencies, respectively, it follows that $j$ is the value of an excedance of $\pi$. 

Next we argue that the values of the excedances of $\pi$ which are not bumped when applying $\RSK$ are exactly the positions of the left endpoints of the right-across tunnels of $D$. This follows by the argument in the proof of \cite[Lemma 6]{EP1}, which states that the number of right-across tunnels of $D$ equals the number of unmatched excedances of $\pi$ in a certain partial matching between the values of the excedances and weak deficiencies of $\pi$. This is proved by showing  that left endpoints of right-across tunnels correspond to values of unmatched excedances of $\pi$, which are exactly the values of excedances which are not bumped when applying $\RSK$. 

Combining the above two cases, we deduce that, if $j$ is the value of an excedance of $\pi$, then step $j$ in $D$ is the right endpoint of a left tunnel or the left endpoint of a right-across tunnel. In both cases, by the argument in the first paragraph of the proof, $j$ is a bottom of $\Match(D)$.

By considering $\pi^{-1}$, we can make a similar claim for the positions of the deficiencies, as inversion of permutations just swaps the values of the excedances with the positions of the deficiencies. 
By Lemma~\ref{lem:RSKDsym}, $\RSKD(\pi^{-1})$ is obtained by reflecting $D$ along its midpoint. Thus, the above argument allows us to conclude that if $i$ is the position of a deficiency of $\pi$, then $2n-i+1$ is the left endpoint of a right tunnel or the right endpoint of a left-across tunnel in $D$, hence a bottom in $\Match(D)$.

Finally, it is shown in \cite[Proposition 3]{EP1} that $i$ is a fixed point of $\pi$ if and only if $(i,2n+1-i)$ is a centered tunnel of $D$, in which case $2n+1-i$ is a bottom in $\Match(D)$. As weak deficiencies consist of deficiencies and fixed points, the result is shown.
\end{proof}

\begin{proof}[Proof of Theorem \ref{ASTConnection}]
     We prove the equivalent statement $\AST\circ \Exc = \Match\circ\RSKD$. For any $\pi \in \S_n(321)$, the matchings $\AST(\Exc(\pi))$ and $\Match(\RSKD(\pi))$ are noncrossing matchings with the same set of bottoms by Lemmas~\ref{PermViaAST} and~\ref{PermViaRSKD}, so they are the same matching by Lemma~\ref{lem:UniqueMatchingData}. 
\end{proof}

\subsection{$\AST$ in type $B$}

Let $\B^m$ denote the positive root poset of the type $B_m$ root system. The poset $\B^m$ can be described as the set of intervals $\{[i,j]: 1\le i\le j\le 2m-1,\, i+j\le 2m\}$ ordered by inclusion, and it is isomorphic to the quotient of $\A^{2m-1}$ by the relations of $[i,j]\sim [2m-j,2m-i]$ for all $[i,j]\in \A^{2m-1}$.  In particular, we will identify the set $\cA(\B^m)$ with the subset of $\cA(\A^{2m-1})$ consisting of those antichains that are invariant under vertical reflection, by associating each $A\in \cA(\B^m)$ to the antichain
\begin{equation}\label{eq:hatA}
\hat{A}=\{[i,j]:[i,j]\in A\text{ or }[2m-j,2m-i]\in A\}\in\cA(\A^{2m-1}).
\end{equation}

The Armstrong--Stump--Thomas bijection in type $B$ \cite{AST13} is the map $\AST_B$ defined for any $A\in\cA(\B^m)$ by $\AST_B(A)=\AST(\hat{A})$, where $\AST$ is the map from Definition~\ref{def:AST}. See Figure~\ref{fig:ASTB} for an example. 

\begin{figure}[h]
\centering
\begin{tikzpicture}
\typeB{3}
\an{1}{2} 
\draw (.4,.6) node[left,scale=.7] {$[1,2]$};
\draw[->] (3.5,1) -- (4,1);
 \draw[->] (3.5,3.5) to [bend left =15] node[above]{$\AST_B$} (12,3.5);
\begin{scope}[shift = {(4,0)}]
    \typeA{6}
    \an{1}{2}
    \an{6}{7}
    \draw (.4,.6) node[left,scale=.7] {$[1,2]$};
     \draw (5.6,.6) node[right,scale=.7] {$[6,7]$};
  \draw[->] (6, 1) --node[above]{$\AST$} (6.5,1);
\end{scope}
\begin{scope}[shift = {(13,1.5)}]
\foreach \i in {1,...,16} {
\filldraw (90+11.25-\i*22.5:1.6) coordinate (v\i) circle (.1);
}
\foreach \i in {1,...,8}
{
\draw (90+11.25-\i*22.5:1.9) node[scale=.7] {\i};
\draw (90-11.25+\i*22.5:1.9) node[scale=.7] {$\ol{\i}$};
}
\matc{2}{3}\matc{7}{8}\matc{15}{16}\matc{1}{14}\matc{4}{13}\matc{5}{12}\matc{10}{11}\matc{6}{9}
\end{scope}
\end{tikzpicture}
\caption{An example of $\AST_B$ applied to the antichain $\{[1,2]\}$ in $\B^4$.}
\label{fig:ASTB}
\end{figure}
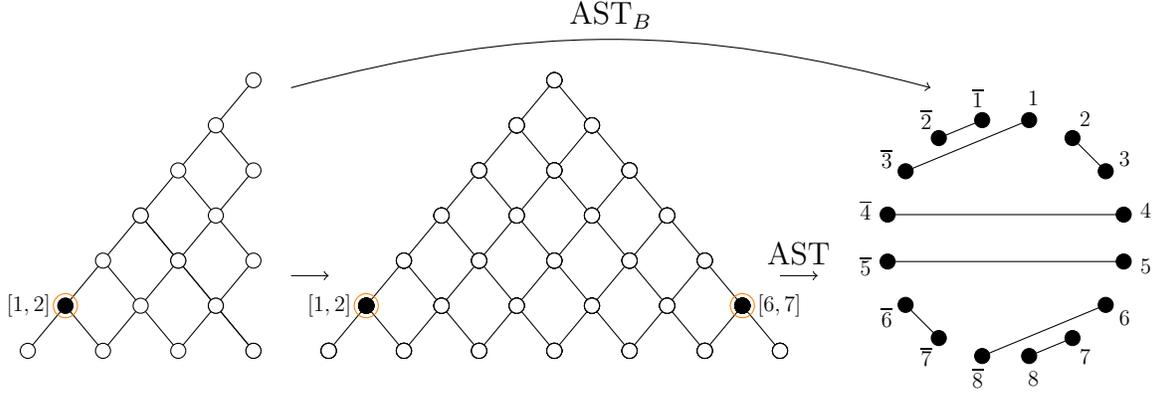

It is shown in \cite[Lemma 3.5]{AST13} that the map $\rA^n$ on $\cA(\A^{n-1})$, obtained by applying rowmotion $n$ times, sends an antichain $\{[i_1,j_1],\dots,[i_k,j_k]\}$ to the antichain $\{[n-j_1,n-i_1],[n-j_2,n-i_2],\dots,[n-j_k,n-i_k]\}$. 
The antichains that are invariant under this map are those that are symmetric under vertical reflection. When $n=2m$, this set was identified with $\cA(\B^m)$ via the embedding~\eqref{eq:hatA}. Thus, by the equivariant property~\eqref{eq:equivariantAST}, the image of $\AST_B$ consists of the matchings in $\NC_{2m}$ that are invariant under applying $\Rot^{2m}$, i.e., rotation by $180^\circ$. 
Matchings with this property are called {\em centrally symmetric}.
Denoting by $\CSNC_n$ the set of centrally symmetric matchings in $\NC_n$, one concludes that $\AST_B$ is a bijection between $\cA(\B^m)$ and $\CSNC_{2m}$.

It this section we give a new proof of this property of $\AST_B$. Unlike the proof in~\cite{AST13}, which relies on the properties of rowmotion, our proof follows from the alternative description of $\AST$ in terms of $\RSK$ gven in Theorem~\ref{ASTConnection}, using well-known properties of the $\RSK$ correspondence.

For $\pi\in \S_n$, its \emph{reverse-complement} is the permutation $\pi^{rc}$ such that $\pi^{rc}(i)=n+1-\pi(n+1-i)$ for all $i$. The array  of $\pi^{rc}$ is obtained by rotating the array of $\pi$ by $180^\circ$. As taking the inverse of $\pi$ corresponds to reflecting its array along the main diagonal, it follows that the array of $(\pi^{rc})^{-1}$ is obtained by reflecting the array of $\pi$ along the secondary diagonal, i.e., the one passing through the top-left and bottom-right corners of the array. 

Via the bijection $\Exc$ from Definition~\ref{def:Exc}, reflection of the array of a permutation along its secondary diagonal translates into a vertical reflection of the corresponding antichain. Thus, $\Exc$ restricts to a bijection between permutations in $\S_n(321)$ that are invariant under reflection along the secondary diagonal of the array, and antichains in $\cA(\A^{n-1})$ that are invariant under vertical reflection. 
When $n=2m$, we obtain a bijection
\begin{equation} \label{eq:Exc-rc}
\Exc:\{\pi\in\S_{2m}(321):\pi=(\pi^{rc})^{-1}\}\to\cA(\B^{m}).
\end{equation}

The goal of this section is to prove the following result.

\begin{theorem}\label{CentralSymmetry}
The map $\Match\circ\RSKD$ restricts to a bijection between $\{\pi\in\S_n(321):\pi=(\pi^{rc})^{-1}\}$ and $\CSNC_n$.
Thus, the map $\AST_B$ is a bijection between $\cA(\B^m)$ and $\CSNC_{2m}$.
\end{theorem}

We will prove this theorem using properties of the Robinson--Schensted--Knuth correspondence. The behavior of $\RSK$ under reverse-complementation and inversion of permutations is well understood.

\begin{theorem}[{\cite[Thm.\ A1.2.10]{EC2}}] \label{thm:RSKrc}
Let $\pi\in\S_n$ and suppose that $\RSK(\pi) = (P,Q)$. Then $\RSK(\pi^{rc})= (\evac(P),\evac(Q))$, and so $\RSK((\pi^{rc})^{-1}) = (\evac(Q),\evac(P))$. 
\end{theorem}

Denote by $\SYT_n^{\le2}$ the set of standard Young tableaux with $n$ boxes and at most two rows.
Generalizing the bijection $\Tab$ from Section~\ref{sec:promotion}, we can view any $T\in\SYT_n^{\le2}$ as a lattice path that stays weakly above the diagonal, where the $i$th step is a $\uu$ or a $\dd$ depending on whether $i$ is in the top or the bottom row of $T$, respectively. By generalizing the map in Definition~\ref{def:Match}, this path induces a partial matching on $\{1,2,\dots,n\}$ points, where two points are matched if the steps in those positions of the path form a tunnel. Denote this partial matching by $M(T)$. Note that $i$ is an unmatched vertex in $M(T)$ if and only if the $i$th step of the path is a $\uu$ step that is at a lower height that all the steps to its right.

\begin{lemma}\label{EvacLemma}
    For any $T\in\SYT_n^{\le2}$ , the partial matching $M(\evac(T))$ is obtained by swapping the roles of vertices $k$ and $n+1-k$, for $1\le k\le n$, in the partial matching $M(T)$. 
\end{lemma}

See Figure~\ref{fig:M} for an example.

\begin{figure}[htb]
\centering
\begin{tikzpicture}
\node[left] at (-.5,1) {$T = \young(123679,458\ten)$};
\coordinate (a) at (1,0);
\coordinate (b) at (.5,.5);
\filldraw[thick] (0,0)\start\up\up\up\dn\dn\up\up\dn\up\dn;
\begin{scope}[shift={(6,1)}]
\foreach \i in {1,...,10} {
\filldraw (90+9-\i*18:1.4) coordinate (v\i) circle (.1);
\draw (90+9-\i*18:1.7) node[scale=.7] {\i};
}
\matc{2}{5}\matc{3}{4}\matc{7}{8}\matc{9}{10}
\draw (2,0) node[right] {$M(T)$};
\end{scope}
\begin{scope} [shift={(0,-4)}]
\node[left] at (-.5,1) {$\evac(T) = \young(12567\ten,2489)$};
\coordinate (a) at (1,0);
\coordinate (b) at (.5,.5);
\filldraw[thick] (0,0)\start\up\dn\up\dn\up\up\up\dn\dn\up;
\begin{scope}[shift={(6,1)}]
\foreach \i in {1,...,10} {
\filldraw (90+9-\i*18:1.4) coordinate (v\i) circle (.1);
\draw (90+9-\i*18:1.7) node[scale=.7] {\i};
}
\matc{1}{2}\matc{3}{4}\matc{6}{9}\matc{7}{8}
\draw (2,0) node[right] {$M(\evac(T))$};
\end{scope}
\end{scope}
\end{tikzpicture}
    \caption{A standard Young tableau with two rows (top) and its evacuation (bottom), along with their associated paths and partial matchings.}
    \label{fig:M}
\end{figure}
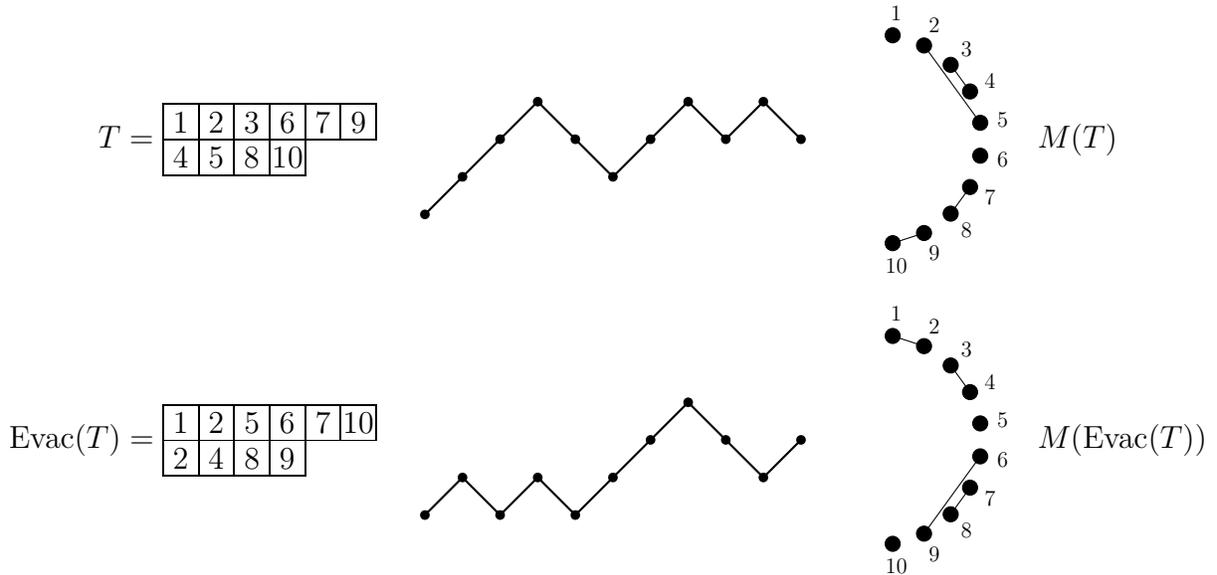

\begin{proof}
We use the fact that $\evac(T)=\jdt(\widetilde{T})$, as given by Lemma~\ref{lem:EvacDescription}.
We proceed by induction on $n$. The lemma trivially holds for $n=1$, so we assume that $n>1$. 

If $T$ has rectangular shape, then $\evac(T)=\widetilde{T}$, with no rectification needed. In this case, the lattice path corresponding to $\evac(T)$ is obtained by reflecting the lattice path of $T$ along a vertical line, and so $(x,y)$ is an arc in $M(T)$ if and only if $(n+1-y,n+1-x)$ is an arc in $M(\evac(T))$, proving the statement.

If $T$ is not rectangular, then $M(T)$ has unmatched vertices. Let $s$ be the smallest one. 
Note that the $s$th step of the associated lattice path is a $\uu$ step, and it is the last step of the path that touches the $x$-axis.

Suppose first that $s=1$. The above property of the $s$th step implies that, when removing the entry $1$ in $T$ and sliding the top row one position to the left, the resulting tableau still has increasing columns. Let $T'$ denote this tableau after subtracting one from its entries so that they range from $1$ to $n-1$.
When rectifying $\widetilde{T}$ starting at its inner corner, the top row will slide one position to the left, and then the entry $n$ will slide up to the first row. At this point in the rectification process, this tableau is precisely $\widetilde{T'}$ with the entry $n$ appended to the first row. Thus, when the  process is complete, the resulting tableau $\evac(T)$ consists of $\evac(T')$ with the entry $n$ appended to the first row.

By the induction hypothesis, $M(\evac(T'))$ is obtained from $M(T')$ by swapping the roles of $k$ and $n-k$ for all $k$. Since $M(T)$ is obtained from $M(T')$ by increasing the labels of all the vertices by one and adding an unmatched vertex labeled $1$, and $M(\evac(T))$ is obtained from $M(\evac(T'))$ by adding an unmatched vertex labeled $n$, it follows that $M(\evac(T))$ is obtained from $M(T)$ by swapping the roles of $k$ and $n+1-k$ for all $k$, completing the proof of this case.

Suppose now that $s>1$. The fact that $s$ is the smallest unmatched vertex of $M(T)$ implies that the subtableau $T_{<s}$ of $T$  consisting of the entries less than $s$ is a rectangular tableau. Denote by $T_{\ge s}$ be the subtableau of $T$ consisting of the entries greater than or equal to $s$, after subtracting $s-1$ from each entry so that they range from $1$ to $n-s+1$.
The matching $M(T)$ consists of the perfect matching $M(T_{<s})$ on the vertices $1,2,\dots,s-1$, together with the matching $M(T_{\ge s})$, with the labels shifted by $s-1$, on the vertices $s,s+1,\dots,n$.

Similarly, in $\widetilde{T}$, the entries larger than $n+1-s$ form a rectangular subtableau, and the entries less than or equal to $n+1-s$ form the subtableau $\widetilde{T_{\ge s}}$.
As explained in the proof of \cite[A.1.2.10]{EC2}, $\jdt$ commutes with removing entries that are less than (or larger than) an arbitrary threshold value. Thus, when applying $\jdt$ to $\widetilde{T}$ in order to obtain $\evac(T)$, the entries larger than $n+1-s$ do not change rows, whereas the entries less than or equal to $n+1-s$ form the subtableau $\evac(T_{\ge s})$. See Figure~\ref{fig:s>1} for an illustration.

\begin{figure}[htb]
\centering
\begin{tikzpicture}[scale=.6]
\node[left] at (-.3,1) {$T=$};
\filldraw[blue!20] (0,0) rectangle (3,1);
\filldraw[blue!40] (0,1) rectangle (3,2);
\filldraw[red!20] (3,0) -- (5,0) -- (5,1) -- (7,1) -- (7,2) -- (3,2) -- (3,0);
\draw (1.5,1) node {$T_{<s}$};
\draw (4,1) node {$T_{\ge s}$};
\begin{scope}[shift={(0,-3)}]
\node[left] at (-.3,1) {$\widetilde{T}=$};
\filldraw[blue!40] (4,0) rectangle (7,1);
\filldraw[blue!20] (4,1) rectangle (7,2);
\filldraw[red!20] (0,0) -- (4,0) -- (4,2) -- (2,2) -- (2,1) -- (0,1) -- (0,0);
\draw (3,1) node {$\widetilde{T_{\ge s}}$};
\draw[->] (8,1)-- node[above,scale=.8]{$\jdt$} (10.5,1);
\end{scope}
\begin{scope}[shift={(15,-3)}]
\node[left] at (-.3,1) {$\evac(T)=$};
\filldraw[blue!40] (2,0) rectangle (5,1);
\filldraw[blue!20] (4,1) rectangle (7,2);
\filldraw[red!20] (0,0) -- (2,0) -- (2,1) -- (4,1) -- (4,2) -- (0,2) -- (0,0);
\draw (2,1.5) node {$\evac(T_{\ge s})$};
\end{scope}

\end{tikzpicture}
    \caption{An illustration of the case $s>1$ in the proof of Lemma~\ref{EvacLemma}.}
    \label{fig:s>1}
\end{figure}
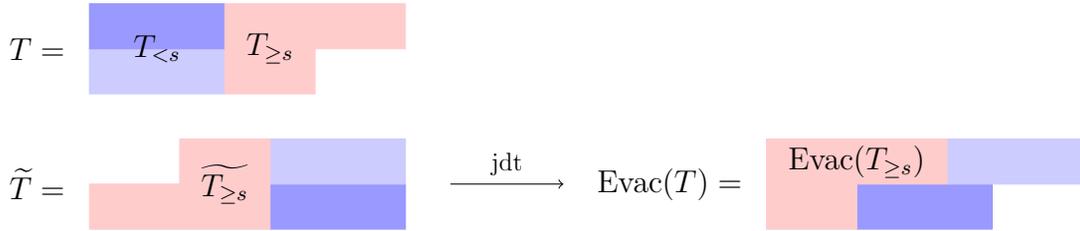

It follows that, for $1\le x,y<s$, the pair $(x,y)$ is an arc in $M(T)$ if and only if $(n+1-y,n+1-x)$ is an arc in $M(\evac(T))$, as in the rectangular case. 
On the other hand, the matching $M(T)$ restricted to the vertices $s,s+1,\dots,n$ equals $M(T_{\ge s})$ with the labels shifted by $s-1$, whereas the matching $M(\evac(T))$ restricted to the vertices $1,2,\dots,n+1-s$ is simply $M(\evac(T_{\ge s})$. By the induction hypothesis, $M(\evac(T_{\ge s}))$ is obtained from $M(T_{\ge s})$ by swapping the roles of $k$ and $n-s+2-k$ for all $k$. We conclude that $M(\evac(T))$ is obtained from $M(T)$ by swapping the roles of $k$ and $n+1-k$ for all $k$, completing the proof.
\end{proof}

\begin{proof}[Proof of Theorem \ref{CentralSymmetry}]
Let $\pi\in\S_n(321)$. If $\RSK(\pi)=(P,Q)$, we know by Theorem~\ref{thm:RSKrc} that $\RSK(\pi^{rc})=(\evac(P),\evac(Q))$. 
Let us show that applying evacuation to each of the tableaux corresponds to reflecting the resulting matching across a horizontal line, so that $\Match(\RSKD(\pi^{rc}))$ is a horizontal reflection of $\Match(\RSKD(\pi))$.

First observe that $\Match(\RSKD(\pi))$ is obtained by first drawing the arcs of $M(P)$, as defined above Lemma~\ref{EvacLemma}, on the vertices $1,2,\dots, n$, then drawing the arcs of $M(Q)$ on the vertices $\ol{1},\ol{2},\dots, \ol{n}$, and finally pairing up the unmatched vertices in $\{1,2,\dots, n\}$ with the unmatched vertices in $\{\ol{1},\ol{2},\dots, \ol{n}\}$ in increasing order.

By Lemma~\ref{EvacLemma}, the partial matching $M(\evac(P))$ is obtained from $M(P)$ by swapping the roles of vertices $k$ and $n+1-k$ for all $1\le k\le n$, and similarly $M(\evac(Q))$ is obtained from $M(Q)$ by swapping the roles of $\ol{k}$ and $\ol{n+1-k}$. In the matching $\Match(\RSKD(\pi^{rc}))$, the vertices that were unmatched in $M(\evac(P))$ and $M(\evac(Q))$  form the same pairs (up to reflection) as in  $\Match(\RSKD(\pi))$. We conclude that $\Match(\RSKD(\pi^{rc}))$ is obtained by reflecting $\Match(\RSKD(\pi))$ across a horizontal line.

Using that $\RSK((\pi^{rc})^{-1}) = (\evac(Q),\evac(P))$ by Theorem~\ref{thm:RSKrc}, and noting that swapping the two tableaux corresponds to reflecting the resulting matching across a vertical line, we conclude that the matching $\Match(\RSKD((\pi^{rc})^{-1}))$ is obtained from $\Match(\RSKD(\pi))$ by applying both a horizontal and vertical reflection, or equivalently, a $180^\circ$ rotation. In particular, $\Match\circ\RSKD$ restrcits to a bijection between permutations $\pi\in\S_n(321)$ satisfying that $\pi=(\pi^{rc})^{-1}$ and matchings in $\NC_n$ that are invariant under $180^\circ$ rotation, proving the first sentence in the statement of the theorem.
The second sentence follows now from Theorem~\ref{ASTConnection} and Equation~\eqref{eq:Exc-rc}, where $n=2m$.
\end{proof}

To conclude, let us show how these ideas can be used to enumerate antichains of $\B^m$ which are fixed under the action of $\LKA\circ \rA$, answering a question of Hopkins and Joseph \cite[Remark 6.7]{HJP}.  
In the following statement, the set $\cA(\B^m)$ is identified with the set of symmetric antichains in $\cA(\A^{2m-1})$ via the map $A\mapsto\hat{A}$ from Equation~\eqref{eq:hatA}. 

\begin{proposition}\label{prop:TypeB_LKA_fixed}
\[|\{A\in \cA(\B^m): \LKA(\rA(A)) = A\}| = 2^m.\]
\end{proposition}

\begin{proof}
We claim that the map $\Exc$ from Definition~\ref{def:Exc} restricts to a bijection
\begin{equation}\label{eq:Exc-rc-inv}
\Exc:\{\pi\in\S_{2m}(321):\pi=(\pi^{rc})^{-1}\text{ and }\pi^{-1}=\pi\}\to\{A\in \cA(\B^m): \LKA(\rA(A)) = A\}.
\end{equation}

Let $\pi\in\S_{2m}(321)$ and $A=\Exc(\pi)$. By Equation~\eqref{eq:Exc-rc}, the condition $\pi=(\pi^{rc})^{-1}$ is equivalent to the fact that $A\in \cA(\B^m)$, so it suffices to show that the condition $\pi=\pi^{-1}$ is equivalent to the fact that $\LKA(\rA(A)) = A$. 
By Equation~\eqref{eq:LKS_LKA_EXC}, the maps $\LKA\circ \rA$ and $\rS\circ\LKS$ are conjugate via $\Exc$, and so $\LKA(\rA(A)) = A$ if and only $\pi=\rS(\LKS(\pi))=\pi^{-1}$, where the last equality follows from Lemma~\ref{LK_of_Perm}.

Next, observe that the left-hand side of Equation~\eqref{eq:Exc-rc-inv} consists of $321$-permutations whose array is symmetric with respect to reflection across both the main and the secondary diagonals. Applying reversal, these are in bijection with $123$-avoiding permutations whose array is symmetric with respect to both diagonals; equivalently, $123$-avoiding involutions that are fixed under reverse-complementation. It is known~\cite[Thm.\ 3.8 (iv)]{EGGE07} that the number of such involutions is $2^m$.
\end{proof}

\section{Future work}
To each poset $\P$ one can associate its chain polytope $\mathcal{C}(\P)$, which is a polytope in $\mathbb{R}^\P$ whose vertices correspond to the antichains of $\P$, see \cite{RS86} for definitions. Rowmotion on antichains of $\P$ can then be extended to a piecewise-linear map on $\mathcal{C}(\P)$, and then detropicalized (by replacing sums with products, and maxima with sums) into a birational map, see \cite{HJ20} and \cite{EP21} for details, as well as for the definitions of homomesy in these settings.

We showed in Section~\ref{homomesies} that the statistics $\fp$ and $\ell_i$ on $321$-avoiding permutations can be translated into statistics on antichains of $\A^{n-1}$. Despite being less natural than their permutation counterparts, these statistics are, by construction, homomesic under antichain rowmotion. Equations~\eqref{eq:fpA} and \eqref{eq:liA} express these statistics as linear combinations of indicator functions and of maxima or minima of indicator functions of antichain elements. As such, they can be interpreted as statistics on the chain polytope $\mathcal{C}(\A^{n-1})$, and furthermore they can be detropicalized and lifted to the birational setting.

We conjecture that, when lifted to the piecewise-linear or birational settings, the statistics $\fp$ and $\ell_i$ from Equations~\eqref{eq:fpA} and \eqref{eq:liA} still exhibit homomesy. This would generalize Theorems~\ref{thm:fp} and~\ref{thm:li}. Computationally, we have verified this claim for $\A^{n-1}$ with $n\le 6$, for all values of~$i$.

In a different direction, after having given a description in Section~\ref{sec:AST} of the Armstrong--Stump--Thomas bijection in types $A$ and $B$ in terms of the $\RSK$ correspondence, one could ask whether there is a related description of $\AST$ in type $D$ using permutations and tableaux.

\subsection*{Acknowledgements} We thank Sam Hopkins, Michael Joseph, Tom Roby, Jessica Striker, and Justin Troyka for helpful discussions.

\end{document}